%% file: Jasour.tex
\documentclass[journal,transmag]{IEEEtran}{\onecolumn}
\ifCLASSINFOpdf
 
\else
  
\fi

\usepackage{fullpage,enumerate,setspace,changepage,multirow,rotating,xcolor}
\usepackage{graphicx,amsmath,amssymb,amsfonts,url, algorithm, algorithmic, etoolbox}
\usepackage{array,subfigure}
\usepackage[small]{caption}

\AtBeginEnvironment{align}{\setcounter{subeqn}{0}}
\newcounter{subeqn} \renewcommand{\thesubeqn}{\theequation\alph{subeqn}}%
\newcommand{\subeqn}{%
  \refstepcounter{subeqn}
  \tag{\thesubeqn}
}
\input defs.tex

\newtheorem{lemma}{Lemma}
\newtheorem{proof}{Proof}
\newtheorem{Theorem}{\textbf{Theorem}}

\def\norm#1{\|#1\|}

\def\norm#1{\|#1\|}

\newcommand{\seq}{\reals^{\mathbb{N}}}
\makeatletter
\newcommand{\thickhline}{%
    \noalign {\ifnum 0=`}\fi \hrule height 1pt
    \futurelet \reserved@a \@xhline
}

\begin{document}


\title{\textsc{Convex Constrained Semialgebraic Volume Optimization: Application in Systems and Control}}

\author{\IEEEauthorblockN{Ashkan Jasour,
Constantino Lagoa}
\IEEEauthorblockA{\IEEEauthorrefmark{} School of Electrical Engineering and Computer Science,
Pennsylvania State University}

}


\IEEEtitleabstractindextext{%
\begin{abstract}

\textit{Abstract}- In this paper, we generalize the chance optimization problems and introduce \textit{constrained volume optimization} where enables us to obtain convex formulation for challenging problems in systems and control. We show that many different problems can be cast as a particular cases of this framework.
In constrained volume optimization, we aim at maximizing the volume of a semialgebraic set under some semialgebraic constraints. Building on the theory of measures and moments, a sequence of semidefinite programs are provided, whose sequence of optimal values is shown to converge to the optimal value
of the original problem. We show that different problems in the area of systems and control that are known to be nonconvex can be reformulated as special cases of this framework. Particularly, in this work, we address the problems of probabilistic control of uncertain systems as well as inner approximation of region of attraction and invariant sets of polynomial systems. Numerical examples are presented to illustrate the computational performance of the proposed approach.

\end{abstract}

}

\maketitle
\IEEEdisplaynontitleabstractindextext
\IEEEpeerreviewmaketitle



\section{Introduction}
The purpose of the proposed approach is to develop convex tractable relaxations for different
problems in the area of systems and control that are known to be "\textit{hard}". We introduce the so-called constrained
volume optimization and show that many challenging problems can be cast as a particular type of this framework. More precisely, we aim at maximizing the volume of a semialgebraic set under some semialgebraic constraints;
i.e., let $\mathcal{S}_1(a)$ and $\mathcal{S}_2(a)$ be semialgebraic sets described by set of polynomial inequalities as follows

\begin{align} \label{intro_set01}
	\mathcal{S}_1(a):= \left\lbrace  x\in \chi:\  \mathcal{P}_{1j}(x,a)\geq0, j=1,\dots ,o_{1} \right\rbrace 
\end{align}
\begin{align} \label{intro_set02}
	\mathcal{S}_2(a):= \left\lbrace  x\in \chi:\  \mathcal{P}_{2j}(x,a)\geq0, j=1,\dots ,o_{2} \right\rbrace 
\end{align}
where $a$ denotes the vector of design parameters. The objective is to find a parameter vector $a$ such
that maximizes the volume of the set $\mathcal{S}_1(a)$ under the constraint $\mathcal{S}_1(a) \subseteq \mathcal{S}_2(a)$.
More precisely, we aim at solving the following problem
\begin{align} \label{intro_P01}
	\mathbf{P_{vol}^*} := \sup_{a\in\mathcal{A}}~\mbox{vol}_{\mu_x} \mathcal{S}_1(a),\\
	\hbox{s.t.}\quad & \mathcal{S}_1(a) \subseteq \mathcal{S}_2(a)\subeqn
\end{align}
where, $\mbox{vol}_{\mu_x} \mathcal{S}_1(a) = \int_{\mathcal{S}_1(a)} d\mu_x$ is the volume of the set $\mathcal{S}_1(a)$ with respect to a given measure $\mu_x$. Many well-known problems can be formulated as a constrained volume optimization problem. As an example, consider the problem of finding the maximal region of attraction (ROA) set for dynamical systems. For a given polynomial system $\dot{x}=f(x)$, maximal ROA set is the largest set of all initial states whose trajectories converge to the origin. This set can be approximated by level sets of a polynomial Lyapunov function $V(x)$. The level set of Lyapunov function $ \{ x \in \reals^n : 0 \leq V(x) \leq 1 \}$ is ROA set if it is contained in the region described by $\{ x \in \reals^n :  \dot{V}(x) \leq \epsilon \lVert x \rVert^2_2 \}$. By characterizing $V(x)$ with a finite order polynomial with unknown coefficients vector $a$ and defining $\mathcal{S}_1(a):=  \left\{ x\in\reals^n :\   0 \leq V(x,a) \leq 1 \right\} $ and $\mathcal{S}_2(a):=  \left\{ x\in\reals^n :\   \dot{V}(x,a)  \leq \epsilon \lVert x \rVert^2_2 \right\} $, the problem of finding maximal ROA set can be reformulated as a constrained volume optimization problem. 

More details are provided in Section \ref{App} where we reformulate different problems in system and control area as constrained volume optimization problems. More precisely, we address the problems of probabilistic control of uncertain systems, inner approximation of region of attraction set and invariant set of polynomial systems, and we also introduce generalized sum of squares problems. The defined constrained volume optimization problem in this work is in general non-convex optimization problem. In this work, relying on measures and moments theory as well as sum of squares theory, we provide a sequence of convex relaxations whose solution converge to the solution of the original problem.

\subsection{Previous Work}
Several approaches have been proposed to solve the problem of approximating the volume of a given set. For example, in \cite{Henrion2009} using theory of measures and moments, hierarchy of SDP problems are proposed to compute the volume of a given compact semialgebraic set. In this method, volume problem can be reformulated as a maximization problem over finite Borel measures supported on the given set and restricted by the Lebesgue measure on a simple set containing the semialgebraic set of interest. 	

As a particular case of the setting in this paper, there exist some approaches to approximate the probability of a given set. In (\cite{Calafiore2006a,Nemirovskia, Nemirovski}) a method called scenario approach to find a tractable approximation for chance constraints problems are provided where probabilistic constraints are replaced by a large number of deterministic constraints.
In (\cite{Nemirovski2007,Pintr1989}), to solve the problems with convex constraints that are affine in random vector, a method called Bernstein approximation is proposed where a convex conservative approximation of chance constraints is constructed using generating functions.  In (\cite{Bertsimas2004,Nemirovski2003}), an alternative approach is proposed where determines an upper bound on the probability of constraints. This method can only be applied to specific uncertainty structures.
In (\cite{Dabbene2009,Feng2009,Feng2011}) convex relaxations of chance constrained problems are presented by introducing the concept of polynomial kinship function to estimate an upper bound on the probability of constraint. It is shown that as the degree of the polynomial kinship function increases, solutions to the relaxed problem  converges to a solution of the original problem. Using the theory of moments, in (\cite{Lagoa2008,Feng2011}), an equivalent convex formulation is provided, where the probability is approximated by computing polynomial approximations of indicator functions.

Building on theory of measures and moments as well as theory of sum of squares polynomials, many approaches have been proposed to reformulate different problems in the area of control and system as a convex optimization problems with Linear Matrix Inequalities (LMI). The methods in (\cite{Korda2013,Korda2013a,Henrion2014,AMJasour2016}), provide hierarchy of finite dimensional
LMI relaxations to compute polynomial outer approximation of the region of attraction
set and maximum controlled invariant set for polynomial
systems. The concept of occupation measure is used to reformulate
the original problem to truncated moment problem. It is shown that the optimal value of the provided LMI
converges to the volume of desired sets and on the other hand the optimal value of the dual problem in continuous
function space converges to the polynomial outer approximation of the set. This proposed method is modified in \cite{Korda2013a}
to obtain a inner approximation for region of attraction set for finite time-horizon polynomial systems. In this case, outer
approximations of complement of region of attraction set is computed. In (\cite{Majumdar,Majumdar2014a}) sum of square formulation to find a suitable Lyapunov function for dynamical system and approximation of region of attraction set is provided. In provided approach, one needs to look for SOS Lyapunov function whose negative derivative is also SOS. Also, a repetitive control design approach is provided where enables to maximize the ROA set of the system.

The purpose of this work, is to develop convex tractable relaxations for different
problems in the area of systems and control. In this paper, we introduce constrained volume optimization problem and show that many problems can be cast as a particular case of this problem.

In our previous work, we addressed the chance constrained optimization problems where one aims at maximizing the probability of a set defined by polynomial inequalities, (see \cite{Jasour2015,Jasour2012,Jasour2013,AMJasour2016}). Based on the theory of measures and moments, we developed a sequence of convex relaxations whose sequence of optimal values converge to the optimal value of the original problem. In the provided convex problem, we look for the nonnegative Borel measure with maximum possible mass on the given semialgebraic set, while simultaneously searching for an upper bound probability measure over a simple set containing the semialgebraic set and restricting the Borel measure. 

In this paper, we generalize the obtained results in chance constrained problem and introduce the so-called constrained semialgebraic volume optimization problem. This framework enables us to find a convex equivalent problem for different problems with deterministic or probabilistic nature. To develop a sequence of convex SDP problems, we use results on moments of measures, known as Lasserre's hierarchy (e.g., see \cite{Lasserre,Lasserre2010,Laurent2009,AMJasour2014}), as well as duality theory which results in sum of squares (SOS) polynomial relaxations (e.g., see \cite{Lasserre2007,Laurent2009}).




\subsection{The Sequel}
The outline of the paper is as follows. In Section II, the notation adopted in the paper and preliminary results on measure and polynomial theory are presented. In Section \ref{Formu}, we precisely define the constrained volume optimization problem with respect to semialgebraic constraints. In Section \ref{App}, some well-known nonconvex problems in system and control area are reformulated as constrained volume optimization problems. In Sections \ref{sec:Primal}, we propose equivalent convex problem and sequences of SDP relaxations to the original problem and show that the sequence of optimal solutions to SDP relaxations converges to the solutions of the original problems. In Section \ref{sec:Dual}, the problems dual to the convex problems given in Section \ref{sec:Primal} are provided. In Section \ref{sec:Res}, some numerical results are presented to illustrate the numerical performance of the proposed approach, and finally, conclusion is stated in Section \ref{sec:Con}.


\section{ Notation and Preliminary Results}\label{Notation}

\label{sec:definitions}

In this paper, building on the theory of measure and moments as well as theory of polynomials, we develop our semidefinite programs to approximate the optimal solution of the original problems. Hence, in this paper the mathematical background and some basic definitions on polynomial and measure theory as well as linear and semidefinite programming are presented.

\subsection{Polynomial Functions}
\label{sec:definitions}

Let $\mathbb{R}[x]$ be the ring of real polynomials in the variables $x \in \mathbb{R}^n$. Given $\cP\in\mathbb{R}[x]$, we represent $\cP$ as $\sum_{\alpha\in\mathbb{N}^n} p_\alpha x^\alpha$ using the standard basis $\{x^\alpha\}_{\alpha\in \mathbb{N}^n}$ of $\mathbb{R}[x]$, and $\mathbf{p}=\{p_\alpha\}_{\alpha\in\mathbb{N}^n}$ denotes the polynomial coefficients. We assume that the elements of the coefficient vector $\mathbf{p}=\{p_\alpha\}_{\alpha\in\mathbb{N}^n}$ are sorted according to grevlex order on the corresponding monomial exponent $\alpha$. Given $n$ and $d$ in $\mathbb{N}$, we define $S_{n,d} := \binom{d+n}{n}$ and $\mathbb{N} ^{\rm n}_d := \{\alpha \in \mathbb N^n : \norm{\alpha}_1 \leq d \}$. Let $\mathbb R_{\rm d}[x] \subset \mathbb R [x]$ denote the set of polynomials of degree at most $d\in \mathbb{N}$, which is indeed a vector space of dimension $S_{n,d}$. Similarly to $\cP\in\mathbb{R}[x]$, given $\cP\in\mathbb R_{\rm d}[x]$, $\mathbf{p}=\{p_\alpha\}_{\alpha\in\mathbb{N}^{\rm n}_d}$ is sorted such that $\mathbb{N} ^{\rm n}_d\ni\mathbf{0} = \alpha ^{(1)} <_g \ldots <_g \alpha ^{(S_{n,d})}$, where $S_{n,d}$ is the number of components in $\mathbf{p}$. 

Now, consider the following definitions on polynomials.

\textbf{Sum of Squares Polynomials:} Let $\mathbb{S}^2[x] \subset \mathbb R [x]$ be the set of sum of squares (SOS) polynomials. Polynomial  $s:\reals^n\rightarrow\reals$ is an SOS polynomial if it can be written as a sum of \emph{finitely} many squared polynomials, i.e., $s(x)= \sum_{j=1}^{\ell} h_j(x)^2$ for some $\ell<\infty$ and $h_j\in\reals[x]$ for $1\leq j\leq \ell$, (\cite{Laurent2009,Parrilo2003}). 

\textbf{Quadratic Module}: For a given set of polynomials $\mathcal{P}_j(x) \in \reals[x], j=1,\dots , \ell$, the quadratic module generated by these polynomials is denoted by $\mathcal{QM}(\mathcal{P}_1, \dots, \mathcal{P}_{\ell})\subset \reals[x]$ and defined as (\cite{Lasserre2010,Laurent2009})
\begin{equation}\label{sec1:def1}
\mathcal{QM}(\mathcal{P}_1, \dots, \mathcal{P}_{\ell}):= s_0(x) + \sum_{j=1}^{\ell}s_j(x)\mathcal{P}_j,\ \{s_j\}_{j=0}^{\ell} \subset \mathbb{S}^2[x]
\end{equation}

\textbf{Putinar's property:} A closed semialgebraic set $\cK = \{ x\in \mathbb{R}^n: \mathcal{P}_j(x)\geq0,\ j=1,2,\dots ,\ell\ \}$ defined by polynomials $\mathcal{P}_j\in \mathbb R [x]$ satisfies \emph{Putinar's property}~\cite{putinar1993positive} if there exists $\mathcal{U}\in \mathbb R [x]$ such that $\lbrace x:  \mathcal{U}(x) \geq 0 \rbrace $ is compact and $\mathcal{U} = s_0 + \sum_{j=1}^{\ell} s_j\mathcal{P}_j $ for some SOS polynomials $\{s_j\}_{j=0}^\ell \subset \mathbb{S}^2[x]$ -- see~ \cite{putinar1993positive,Lasserre2010,Laraki2012}. Putinar's property holds if the level set $\lbrace x: \mathcal{P}_j(x) \geq 0 \rbrace$ is compact for some $j$, or if all $ \mathcal{P}_j $
are affine and $\cK $ is compact - see~\cite{Laraki2012}. 
Putinar's property is not a geometric property of the semi-algebraic set $\cK$, but rather an algebraic property related to the representation of the set by its defining polynomials. Hence, if there exits $M>0$ such that the polynomial $\cP_{\ell+1}(x):= M- \Vert x \Vert^2 \geq 0$ for all $x\in\cK$, then the \textit{new representation} of the set $\cK = \{ x\in \mathbb{R}^n: \mathcal{P}_j(x)\geq0,\ j=1,2,\dots ,\ell+1\ \}$ satisfies Putinar's property, \cite{Jasour2015}.


\subsection{ Measures and Moments}


Let $\cM(\chi)$ be the space of finite Borel measures and $\cM_+(\chi)$ be the cone of finite nonnegative Borel measures $\mu$ such that $supp(\mu)\subset\chi$, where $supp(\mu)$ denotes the support of the measure $\mu$; i.e., the smallest closed set that contains all measurable sets with strictly positive $\mu$ measure. Also, let $C\subset\reals^n$, $\Sigma(C)$ denotes the Borel $\sigma$-algebra over $C$. Given two measures $\mu_1$ and $\mu_2$ on a Borel $\sigma $-algebra $\Sigma $, the notation $\mu_1 \preccurlyeq \mu _2$ means  $\mu _1(S)\le \mu_2(S)$ for any set $S \in \Sigma $. Moreover, if $ \mu_1$ and $ \mu_2$ are both measures on Borel $\sigma$-algebras $\Sigma_1$ and $\Sigma_2$, respectively, then $\mu =\mu_1 \times \mu_2$ denotes the product measure satisfying $ \mu(S_1 \times S_2)=\mu_1 (S_1) \mu_2(S_2)$ for any measurable sets $S_1\in \Sigma_1$, $S_2 \in \Sigma_2$ \cite{Henrion2009}.

Let $\seq$ denote the vector space of real sequences. Given $\mathbf{y}=\{y_\alpha\}_{\alpha\in\mathbb{N}^n}\subset\seq$, 
let $L_\mathbf{y}:\reals[x]\rightarrow\reals$ be a linear map defined as (\cite{Lasserre,Lasserre2010})
\begin{small}
	\begin{equation}
	\label{eq:lin_map}
	\cP \quad \mapsto \quad L_\mathbf{y}(\cP)=\sum_{\alpha\in\mathbb{N}^n}p_\alpha y_\alpha, \quad \hbox{where} \quad \cP(x)=\sum_{\alpha\in\mathbb{N}^n} p_\alpha x^\alpha
	\end{equation}
\end{small}
A sequence $\mathbf y = \{ y_ \alpha \}_{\alpha\in\mathbb{N}^n}\in\seq$
is said to have a \emph{representing measure}, if there exists a finite Borel measure $\mu$ on $\reals^n$ such that $y_{\alpha } = \int{x^{\alpha} d\mu }$ for every $\alpha \in \mathbb N ^n$ -- see~(\cite{Lasserre,Lasserre2010}). In this case, $ \mathbf y $ is called the moment sequence of the measure $\mu $.  

Given two square symmetric matrices $A$ and $B$, the notation $A \succcurlyeq 0$ denotes that $A$ is positive semidefinite, and $A\succcurlyeq B$ stands for $A-B$ being positive semidefinite.

\textbf{Moment Matrix:} Given $r\geq 1$ and the sequence $\{y_\alpha\}_{\alpha\in\mathbb{N}^n}$, the moment matrix $M_r({\mathbf y})\in\reals^{S_{n,r}\times S_{n,r}}$, containing all the moments up to order $2r$, is a symmetric matrix and its $(i,j)$-th entry is defined as follows (\cite{Lasserre,Lasserre2010,AMJasour2014}):
\begin{equation}\label{momnt matirx def}
M_r ( \mathbf y )(i,j):= L_{\mathbf y}\left(x^{\alpha^{(i)}+\alpha^{(j)}}\right)=y_{\alpha^{(i)}+\alpha^{(j)}}\ \ \
\end{equation}
where  $1 \leq i,j \leq S_{n,r}$, $\mathbb{N} ^{\rm n}_r\ni\mathbf{0} = \alpha ^{(1)} <_g \ldots <_g \alpha ^{(S_{n,2r})}$ and $S_{n,2r}$ is the number of moments in $\mathbb{R}^n$ up to order $2r$. Let $\cB_r^T=\left[x^{\alpha^{(1)}},\ldots, x^{\alpha^{(S_{n,r})}}\right]^T$ denote the vector comprised of the monomial basis of $\mathbb R_{\rm r}[x]$. Note that the moment matrix can be written as $M_r({\mathbf y}) = L_\mathbf{y}\left(\cB_r \cB_r^T\right)$; here, the linear map $L_\mathbf{y}$ operates componentwise on the matrix of polynomials, $\cB_r \cB_r^T$. For instance, let $r=2$ and $n=2$; the moment matrix containing moments up to order $2r$ is given as
\begin{equation} \label{moment matrix exa}
M_2\left({\mathbf y}\right)=\left[ \begin{array}{c}

\begin{array}{ccc} y_{00} \ | & y_{10} & y_{01}| \end{array}
\begin{array}{ccc} y_{20} & y_{11} & y_{02} \end{array}
\\
\begin{array}{ccc} - & - & - \end{array}
\ \ \ \  \begin{array}{ccc} - & - & - \end{array}
\\

\begin{array}{ccc} y_{10}\ | & y_{20} & y_{11}| \end{array}
\ \begin{array}{ccc} y_{30} & y_{21} & y_{12} \end{array}
\\

\begin{array}{ccc} y_{01}\ | & y_{11} & y_{02}| \end{array}
\ \begin{array}{ccc} y_{21} & y_{12} & y_{03} \end{array}
\\

\begin{array}{ccc} - & - & - \end{array}
\ \ \ \ \  \begin{array}{ccc} - & - & - \end{array}
\\

\begin{array}{ccc} y_{20}\ | & y_{30} & y_{21}| \end{array}
\ \begin{array}{ccc} y_{40} & y_{31} & y_{22} \end{array}
\\

\begin{array}{ccc} y_{11}\ | & y_{21} & y_{12}| \end{array}
\ \begin{array}{ccc} y_{31} & y_{22} & y_{13} \end{array}
\\

\begin{array}{ccc}y_{02}\ | & y_{12} & y_{03}| \end{array}
\ \begin{array}{ccc} y_{22} & y_{13} & y_{04} \end{array}

\end{array}
\right]
\end{equation}

\textbf{Localizing Matrix:} Given a polynomial $\mathcal{P} \in \mathbb R [x]$, let $ \mathbf p = \{ p_{\gamma }\}_{\gamma\in\mathbb{N}^n}$ be its coefficient sequence in standard monomial basis, i.e., $\cP(x)=\sum_{\alpha\in\mathbb{N}^n} p_\alpha x^\alpha$, 
the $(i,j)$-th entry of the \emph{localizing matrix} $M_r(\mathbf{y};\mathcal{P})\in\reals^{S_{n,r}\times S_{n,r}}$ with respect to $\mathbf y $ and $\mathbf p$ is defined as follows (\cite{Lasserre,Lasserre2010,AMJasour2014}):
\begin{small}
	\begin{equation}\label{localization matrix def}
	M_r(\mathbf y;\mathcal{P})(i,j) := L_{\mathbf{y}}\left(\cP x^{\alpha^{(i)}+\alpha^{(j)}}\right)=\sum_{\gamma \in \mathbb N^n} p_{\gamma} y_{\gamma +\alpha^{(i)}+\alpha^{(j)}} 
	\end{equation}
\end{small} where, $1 \leq i,j \leq  S_{n,d}$. Equivalently, $M_r(\mathbf y,\mathcal{P}) = L_{\mathbf{y}}\left(\mathbf \cP \cB_r\cB_r^T\right)$, where $L_{\bf y}$ operates componentwise on $\cP \cB_r\cB_r^T$. For example, given $\mathbf{y}=\{y_\alpha\}_{\alpha\in\mathbb{N}^2}$ and the coefficient sequence $\mathbf{p}=\{p_\alpha\}_{\alpha\in\mathbb{N}^2}$ corresponding to polynomial $\cP$,
\begin{equation}
\mathcal{P}(x_1,x_2)= bx_1-cx^2_2,
\end{equation}
the localizing matrix for $r=1$ is formed as follows
\begin{equation}
M_1(\mathbf{y};\mathcal P)= \begin{small}
\left[ \begin{array}{ccc}
by_{10}-cy_{02} & by_{20}-cy_{12} & by_{11}-cy_{03} \\
by_{20}-cy_{12} & by_{30}-cy_{22} & by_{21}-cy_{13} \\
by_{11}-cy_{03} & by_{21}-cy_{13} & by_{12}-cy_{04} \end{array}
\right]
\end{small}
\end{equation}

\subsection{ Preliminary Results on Measures and Polynomials}

In this section, we state some standard results found in the literature that will be referred to later in this thesis.

\textbf{Moment Condition:} The following lemmas give necessary, and sufficient conditions for sequence of moments $\mathbf y$ to have a representing measure $\mu$ -- for details see \cite{Henrion2009,Lasserre2007,Lasserre2010, Jasour2015}.
\begin{lemma}
	\label{sec2:lem1}
	Let $\mu$ be a finite Borel measure on $\reals^n$, and $\mathbf{y}=\{y_\alpha\}_{\alpha\in\mathbb{N}^n}$ such that $y_\alpha=\int x^\alpha d\mu$ for all $\alpha\in\mathbb{N}^n$. Then $M_d(\mathbf y)\succcurlyeq 0$ for all $ d\in \mathbb N$.
\end{lemma}

\begin{lemma}
	\label{sec2:lem2}
	Let $\mathbf y = \{y_{\alpha}\}_{\alpha\in\mathbb{N}^n}$ be a real sequence. If $M_d({\bf y}) \succcurlyeq 0$ for some $d\geq 1$, then
	\begin{equation*}
	\vert y_{\alpha} \vert \leq \max \left\{ y_0, \max_{i=1,\ldots, n} L_{\mathbf{y}}\left( x_i^{2d}\right) \right\}\quad \forall \alpha\in\mathbb{N}^n_{2d}.
	\end{equation*}
\end{lemma}
\begin{lemma}
	\label{sec2:lem3}
	If there exist a constant $c > 0$ such that $M_d(\mathbf y)\succcurlyeq 0$ and $|y_{\alpha}| \leq c$ for all $ d\in \mathbb N$ and $\alpha\in\mathbb{N}^n$, then there exists a representing measure $\mu$ with support on $[-1, 1]^n$.
\end{lemma}

\begin{lemma}	\label{sec2:lem4}
	Let $\mu$ be a Borel probability measure supported on the hyper-cube $[-1, 1]^n$. Its moment sequence $\mathbf y\in\reals^\mathbb{N}$ satisfies $\Vert \mathbf y \Vert_{\infty} \leq 1$.
\end{lemma}

Given polynomials $\mathcal{P}_j\in \mathbb R [x]$, let $\textbf{p}_j$ be its coefficient sequence in standard monomial basis for $j=1,2,\dots ,\ell$; consider the semialgebraic set $\cK$ defined as
\begin{equation}\label{preliminary result_semi algebraic set}
\cK = \{ x\in \mathbb{R}^n: \mathcal{P}_j(x)\geq0,\ j=1,2,\dots ,\ell\ \}.
\end{equation}
The following lemma gives a necessary and sufficient condition for $\mathbf y$ to have a representing measure $\mu$ supported on $\cK$ -- see \cite{Henrion2009,Lasserre2007,Lasserre2010,Lasserre}.
\begin{lemma}
	\label{sec2:lem5}
	If $\cK$ defined in \eqref{preliminary result_semi algebraic set} satisfies Putinar's property, 
	then the sequence $\mathbf y = \{ y_\alpha\}_{\alpha\in\mathbb{N}^n}$ has a \emph{representing} finite Borel measure $\mu$ on the set $\cK$, if and only if
	\begin{equation*}
	M_d(\mathbf y)\succcurlyeq 0,\quad M_d(\mathbf y;\textbf{p}_j)\succcurlyeq 0,\ \ j=1,\dots ,\ell, \hbox{ for all }  d\in \mathbb N.
	\end{equation*}
\end{lemma}

\textbf{Measure of Compact Set}: The following lemma, proven in~\cite{Henrion2009}, shows that the Borel measure of a compact set is equal to the optimal value of an infinite dimensional LP problem.
\begin{lemma}
	\label{sec2:lem6}
	Let $\Sigma$ be the Borel $\sigma$-algebra on $\reals^n$, and $ \mu_1$ be a measure on a compact set $\cB\subset\Sigma$. Then for any given $\cK\in\Sigma$ such that $\cK\subseteq  \cB$, one has
	\begin{equation*}
	\mu_1(\cK)= \int_{\cK} d\mu_1 = \sup_{\mu_2\in\cM(\cK)} \left\lbrace  \int d\mu_2 : \mu_2 \preccurlyeq \mu_1\right\rbrace,
	\end{equation*}
	where $\cM(\cK)$ is the set of finite Borel measures on $\cK$.
\end{lemma}

\textbf{SOS Representation: }The following lemma gives a sufficient condition for $f \in \mathbb{R}[x]$ to be nonnegative on the set $\mathcal{K}$-- see \cite{anderson1987linear,Lasserre2007,Lasserre,Lasserre2010}.
\begin{lemma}
	\label{sec2:lem7}
	Assume $\cK$ defined in \eqref{preliminary result_semi algebraic set} satisfies Putinar's property. If $\mathcal{P} \in \mathbb{R}[x]$ is strictly positive on $\cK$, then $\mathcal{P} \in \mathcal{QM}(\{ \mathcal{P}_j \}_{j=1}^{\ell})$. Hence,
	\begin{equation*}
	\mathcal{P} = s_0 + \sum_{j=1}^{\ell} s_j \mathcal{P}_j, \ s_j \in \mathbb{S}^2[x], \ j=0,...,\ell
	\end{equation*}
\end{lemma}

\textbf{Duality}: The following theorems show the relationship between measures, continuous functions and polynomials:

\textbf{i)} \textbf{Stone-Weierstrass Theorem:}
Every continuous function defined on a closed set can be uniformly approximated as closely as desired by a polynomial function, \cite{ASH1972}.

\textbf{ii)} \textbf{Riesz Representation Theorem:} Let $\mathcal{C}(\chi)$ be the Banach space of continuous functions on $\chi$ with associated norm $\| f \| := \sup_{x \in \chi} |f(x)|$ for $f \in \mathcal{C}$ and $\mathcal{C}_+(\chi) := \left\lbrace  f \in \mathcal{C} : f \geq 0 \hbox{ on } \chi \right\rbrace $ be the cone of nonnegative continuous functions. The cone of nonnegative measures is dual to the cone of nonnegative continuous functions with inner product $\langle \mu,f \rangle := \int_{\chi} fd\mu, \ \mu \in \mathcal{M}_+(\chi), \ f \in \mathcal{C}_+(\chi)$; i.e., any $\mu \in \cM_+(\chi)$ belongs to the space of all linear functional on $\mathcal{C}_+(\chi)$ - see (\cite{Royden2010}, Section 21.5, \cite{anderson1987linear, A.Barvinok2002}).

\subsection{Linear and Semidefinite Programming }
In this section preliminary results on linear program and semidefinite programs are presented.

Consider the linear programming (LP) problem in standard form
\begin{align}
\mathbf{P^*}:=&\ \max \langle x , c \rangle \label{LP}\\
&\hbox{s.t.}\quad Ax \leq b \label{LP_1}\subeqn\\
&x \geq 0 \label{LP_2}.\subeqn
\end{align}
where, $x \in \reals^n$ is variable vector, $A: \reals^n \rightarrow \reals^m $ is the linear operator, $b \in \reals^m$ are real matrices and vector.Also, $\langle x , c \rangle = c^Tx$. Based on standard results on LP \cite{anderson1987linear,A.Barvinok2002}, the dual problem of \eqref{LP} reads as
\begin{align}
\mathbf{P_{Dual}^*}:=&\ \min \langle b , y \rangle \label{LPD}\\
&\hbox{s.t.}\quad A^*y \geq c \label{LPD_1}\subeqn\\
&y \geq 0 \label{LPD_2}.\subeqn
\end{align}
where, $A^*: \mathbb{R}^{m} \rightarrow \mathbb{R}^n $ denotes the adjoint operator of $A$, i.e., $\langle A^*y,x \rangle = \langle y, Ax \rangle $. The following theorem shows the relationship of primal and dual problems.

\begin{Theorem}
	\label{sec2:Theo1}	
	\textbf{Strong Duality:} If in problem \eqref{LP}, $\langle x , c \rangle$ is finite value and the set $\left\lbrace  (Ax,\langle x,c\rangle) : x \geq 0 ) \right\rbrace $ is closed, then there is no duality gap between \eqref{LP} and \eqref{LPD}, i.e., $\mathbf{P^*}=\mathbf{P_{Dual}^*}$, (\cite{anderson1987linear}, Theorem 3.10, \cite{A.Barvinok2002}, Theorem 7.2)
\end{Theorem}

Consider the semidefinite programming (SDP) problem in standard form
\begin{align}
\mathbf{P^*}:=&\ \min \langle C , X \rangle \label{SDP}\\
&\hbox{s.t.}\quad \langle A_i , X \rangle =b_i,\ i=1,...,m \label{SDP_1}\subeqn\\
&X \succcurlyeq 0 \label{SDP_2}.\subeqn
\end{align}
where,  $A_i,C \in  \reals^n \times \reals^n $, vector $b \in \reals^m$, and $X \in \reals^n \times \reals^n $, $\langle C , X \rangle = trace(CX)$. Based on standard results on SDP \cite{M.Trnovska2005,Todd}, the dual problem of \eqref{SDP} reads as
\begin{align}
\mathbf{P_{Dual}^*}:=&\ \max b^Ty \label{SDPD}\\
&\hbox{s.t.}\quad C - \sum_{i=1}^m A_iy_i \succcurlyeq 0 \label{SDPD1}\subeqn
\end{align}
%
The following theorem shows the relationship of primal and dual problems.

\begin{Theorem}
	\label{sec2:Theo2}	
	\textbf{Slater's sufficient condition:}  if the feasible set of strictly positive matrices in constraint of primal SDP is nonempty, then there is no duality gap between \eqref{SDP} and \eqref{SDPD}, i.e., $\mathbf{P^*}=\mathbf{P_{Dual}^*}$, (\cite{M.Trnovska2005,Todd}) . 
\end{Theorem}


\section{Problem Statement}\label{Formu}

In this work, we consider \emph{constrained volume optimization problems} defined as follows: Let $\left(\chi,\Sigma_x,\mu_x\right)$ be a given measure space with $\Sigma_x$ denoting the Borel $\sigma$-algebra of $\chi \subset \reals^n$ and $\mu_x$ denoting a finite nonnegative Borel measure on $\Sigma_x$. Consider semialgebraic sets $\mathcal{S}_1: \reals^n \rightarrow \Sigma_x $ and $\mathcal{S}_2: \reals^n \rightarrow \Sigma_x $ as follows
\begin{align} \label{intro_set1}
	\mathcal{S}_1(a):= \left\lbrace  x\in \chi:\  \mathcal{P}_{1j}(x,a)\geq0, j=1,\dots ,o_{1} \right\rbrace 
\end{align}
\begin{align} \label{intro_set2}
	\mathcal{S}_2(a):= \left\lbrace  x\in \chi:\  \mathcal{P}_{2j}(x,a)\geq0, j=1,\dots ,o_{2} \right\rbrace 
\end{align}
where $\mathcal{P}_{1j}:{{\mathbb{R}}}^n\times {{\mathbb{R}}}^m\rightarrow{\mathbb{R}}$, $j=1,2,\dots ,o_{1}$, and $\mathcal{P}_{2j}:{{\mathbb{R}}}^n\times {{\mathbb{R}}}^m\rightarrow{\mathbb{R}}$, $j=1,2,\dots ,o_{2}$ are given polynomials. We focus on the following problem.
\begin{align} \label{intro_P1}
	\mathbf{P_{vol}^*} := \sup_{a\in\mathcal{A}}~\mbox{vol}_{\mu_x} \mathcal{S}_1(a),\\
	\hbox{s.t.}\quad & \mathcal{S}_1(a) \subseteq \mathcal{S}_2(a)\subeqn
\end{align}
where, $\mbox{vol}_{\mu_x} \mathcal{S}_1(a) = \int_{\mathcal{S}_1(a)} d\mu_x$ is the volume of the set $\mathcal{S}_1(a)$ with respect to given finite Borel measure $\mu_x$. In the problem \eqref{intro_P1}, we are looking for $a \in \mathcal{A} \subset {{\mathbb{R} }}^m$, the parameters of sets $\mathcal{S}_1$ and $\mathcal{S}_2$, such that volume of the set $\mathcal{S}_1(a)$ becomes maximum while it is contained in the set $\mathcal{S}_2(a)$.



\section{Applications in Systems and Control}\label{App}

In this section we focus on some well-known challenging problems in the area of system and control which are, in general, nonconvex and computationally hard. As an first step in the development of convex relaxations of these problems, we reformulate them as a constrained volume optimization problem. 
\subsection{Region of Attraction} \label{App_Lya}
Consider a continuous-time system of the form
\begin{equation} \label{Sys_C}
	\dot{x} = f(x) 
\end{equation}
where, $f : \mathbb{R}^{n}  \rightarrow \mathbb{R}^n$ is a polynomial function, $x \in \chi \subset \mathbb{R}^n$ are system states and $\chi$ is compact that contains the origin. Let the origin $x = 0$ be an asymptotically stable equilibrium point for the system. The region of attraction (ROA) set $\mathcal{R}_x \subseteq \chi$ is defined as largest set of all initial states whose trajectories converge to the origin. 

For the system in (\ref{Sys_C}), assume there exist a function $V(x)$ such that 
\begin{equation}
	V(0)=0, V(x) > 0 \ \mbox{for x} \neq 0 \ 
\end{equation}
The level set defined as $\mathcal{R} = \left\lbrace   x \in \chi : V(x) \leq r    \right\rbrace $ is an inner approximation of ROA if $\dot{V}(x)< 0 $ for all $x \in \mathcal{R}$ and $\dot{V}(x) = 0 $ for $x = 0$ \cite{khalil2002nonlinear}. In this case function $V(x)$ is a Lyapunov function for the system in (\ref{Sys_C}). We assume that polynomial system in \eqref{Sys_C} admits a polynomial Lyapunov function (see \cite{AliAhmadi} for discussion on existence of a polynomial Lyapunov function). Hence, we can describe it as a finite order polynomial $V(x) = \sum_{\|i\|_1 \leq d} a_i x^i \in \reals_d[s]$, where $a \in \mathcal{A} \subset \reals^{S_{n,d}}$ is a vector of unknown coefficients,  (\cite{Papachristodoulou2002,Majumdara}). The following optimization problem can be used to find $V(x)$ and corresponding inner approximation of maximal ROA. For a given system in \eqref{Sys_C} and given compact sets $\chi$ and $\mathcal{A}$, solve
\begin{align} 
	\mathbf{P_{ROA}^*}:=&\ \max_{a \in \mathcal{A} } \mbox{vol}_{\mu_x}(\mathcal{R})\label{Lya} \\
	\hbox{s.t.}\quad & V(x) = \sum_{\|i\|_1 \leq d} a_i x^i , \ V(0)=0 \subeqn  \label{Lya_1}\\
	&\ V(x) > 0, \ \mbox{for all } x \neq 0 \subeqn  \label{Lya_2}\\
	&\mathcal{R} = \left\lbrace   x \in \chi : \ 0 \leq V(x) \leq r    \right\rbrace  \label{Lya_3} \subeqn \\
	&\mathcal{R} \subseteq  \left\lbrace x \in \chi :  \dot{V}(x) \leq -\epsilon_r \Vert x \Vert_2^2 \right\rbrace \label{Lya_4} \subeqn
\end{align} 
where, $\hbox{vol}_{\mu_x}(\mathcal{R}) = \int_\mathcal{R} d\mu_x $ is the volume of the set $\mathcal{R}$ with respect to Lebesgue measure $\mu_x$ supported on $\chi$, $\Vert . \Vert_2$ is $l_2$ norm and $r > 0$ and $\epsilon_r > 0$ are given constants. By solving problem in \eqref{Lya}, we are in fact looking for a Lyapunov function $V(x)$ among the space of polynomial functions of order at most $d$.
By defining the sets $\mathcal{S}_1(a) = \{ x\in \chi :\ 0 \leq V(x,a) \leq r \} $ and $\mathcal{S}_2(a) = \{x\in \chi : \dot{V}(x,a) \leq - \epsilon_r \Vert x \Vert_2^2 \} $, problem in \eqref{Lya} can be restated as volume optimization problem in \eqref{intro_P1}.
With the same reasoning, one can extend the problem in \eqref{Lya} for discrete-time systems $ x_{k+1} = f(x_k) $ by replacing the derivative of Lyapunov function  $\dot{V}(x) $ with the difference Lyapunov function $ \bigtriangleup V(x) = V(x_{k+1})-V(x_k)$.
\subsection{Maximal Invariant Set} \label{App_Inv}
Consider a discrete time system
\begin{equation} \label{Sys1}
	x_{k+1} = f(x_k)
\end{equation}
where, $f : \mathbb{R}^{n}  \rightarrow \mathbb{R}^{n}$ is a polynomial function and $x_k \in \chi_{ext} \subset \mathbb{R}^{n}$ are system states. Given a compact set $\chi \subset \chi_{ext}$, the set $\mathcal{V} \subset \chi$ is \textit{robustly invariant} if
\begin{equation}\label{InvarSet}
	f(x) \in \mathcal{V}, \hbox{   for all   }  x \in \mathcal{V}
\end{equation}
Hence, \textit{maximal invariant set} is the maximal set of all initial states whose trajectories remains inside the set. The following statement holds true for robustly invariant sets. Consider the set $\mathcal{V} =\{ x\in \chi:\  \mathcal{P}(x)\geq 0  \}$, for bounded above function $\mathcal{P}(x)$. The set $ \mathcal{V}  $ is an robustly invariant set for dynamical system above if $\mathcal{P}(f(x)) \geqslant 0$ for all $x \in \mathcal{V}$.

In order to find a polynomial approximation of maximal robustly invariant set for dynamical system \eqref{Sys1}, we characterize the function $\mathcal{P}(x)$ with finite order polynomial as $\mathcal{P}(x) = \sum_{\|i\|_1 \leq d} a_i x^i \in \reals_d[x] $, where $a \in \mathcal{A} \subset \reals^{S_{n,d}} $ is vector of unknown coefficients. Then, we consider following optimization problem to obtain unknown coefficients.

Assume that the given compact set $\chi$ can be described as $\chi = \left\lbrace   x \in \reals^{n} : \mathcal{P}_{\chi}(x) \geq 0    \right\rbrace$ $\subset \reals^{n}$, for some polynomial $\mathcal{P}_{\chi}(x)$.
Then for a given system in \eqref{Sys1} and given compact sets $\chi \subset \reals^{n}$ and $\mathcal{A} \subset \reals^{S_{n,d}}$, solve
\begin{small}
	\begin{align}
		\mathbf{P_{INV}^*}:=&\ \max_{a \in \mathcal{A} } \mbox{vol}_{\mu_x}(\mathcal{\mathcal{V}})\label{max_pc3} \\
		\hbox{s.t.}\quad & \mathcal{P}(x) = \sum_{\|i\|_1 \leq d} a_i x^i , \subeqn  \label{eq:Inv1}\\
		\mathcal{V} &= \left\lbrace   x \in \chi: \ \mathcal{P}(x) \geq 0   \right\rbrace  \label{eq:Inv3} \subeqn \\
		\mathcal{V} & \subset  \left\lbrace x \in  \chi : \mathcal{P}(f(x)) \geq 0 \right\rbrace \label{eq:DV} \subeqn
	\end{align}
\end{small}where, $\mbox{vol}_{\mu_x}(\mathcal{V}) = \int_\mathcal{V} d\mu_x $ is the volume of the set $\mathcal{V}$ with respect to Lebesgue measure $\mu_x$ supported on $\chi_{ext}$.

By defining the sets $\mathcal{S}_1(a) = \{ x \in \chi : \mathcal{P}(x,a) \geq 0 \} $ and $\mathcal{S}_2(a) = \{x \in \chi : \mathcal{P}(f(x),a) \geq 0 \} $, problem \eqref{max_pc3} can be restated as volume optimization problem \eqref{intro_P1}.

\subsection{Probabilistic Control of Uncertain Systems} \label{App_PC}
Consider the following discrete-time stochastic dynamical system 
\begin{equation}\label{intro_PE1}
	x_{k+1}= f(x_k,u_k,\omega_k,\delta)
\end{equation}
where $f: \reals^{n_x+n_u+n_{\omega}+n_{\delta}} \rightarrow \reals^{n_x}$ is a polynomial function,  $x_k \in \chi \subseteq \reals^{n_x}$ is system state, $u_k \in \psi \subseteq \reals^{n_u}$ is control input, and $\omega_k\in \Omega \subseteq R^{m_\omega}$ is disturbance, at time step $k$, and  $\delta\in \Delta \subseteq \reals^{n_\delta}$ is uncertain model parameter. The initial state $x_0 \in \chi_0 \subseteq \chi$, model parameter $\delta $, and disturbance $\omega_k$ at time $k$ are independent random variables with probability measure $\mu_{x_0}$, $\mu_{\delta}$, and $\mu_{\omega_k}$ supported on $\chi_0 $, $ \Delta $, and $\Omega $, respectively. We assume that  $\chi_0, \Delta$, and $\Omega $ are compact semialgebraic sets of the form
$\chi_0 = \lbrace x \in \chi : \mathcal{P}_0(x) \geq 0  \rbrace$,
$\Delta = \lbrace \delta \in \reals^{n_{\delta}} : \mathcal{P}_{\delta}(\delta) \geq 0 \rbrace$,
$\Omega = \lbrace \omega \in \reals^{n_{\omega}} : \mathcal{P}_{\omega}(\omega) \geq 0 \rbrace$
for given polynomials $ \mathcal{P}_{0}, \mathcal{P}_{\delta}, \mathcal{P}_{\omega}$. Let $N$ be a given integer and $\chi_N$ be given desired terminal set. Also, $\chi_{x_k}$ is given compact feasible set for states of the system at time $k$.

In \textit{probabilistic control}, we aim at finding a polynomial state feedback control input to maximize the probability that $x_k$ belongs to the feasible set $\chi_{x_k}$ and states of system reach the target set $\chi_N$ in at most $N$ steps \cite{Jasour2013}.
We assume that the desired terminal set $\chi_N$ and the feasible sets $\chi_{x_k}$ are defined as the compact semialgebraic sets as 
\begin{equation}\label{Pro_set}
	\chi_N = \lbrace x \in \chi : \mathcal{P}_{N}(x) \geq 0 \rbrace
\end{equation}
\begin{equation}\label{Pro_Fset}
	\chi_{x_k} = \lbrace x \in \chi : \mathcal{P}_{k}(x) \geq 0 \rbrace
\end{equation}
Also, control input $u(x_k): \reals^{n_x} \rightarrow \reals^{n_u}$ takes the state feedback form as 
\begin{equation}\label{Pro_u}
	u_j(x_k) = \sum_{||i||_1\leq d_u} a_{ji} x_k^i , \quad j=1,...,n_u 
\end{equation}
where $u_j: \reals^{n_x} \rightarrow \reals$ is polynomial of order no more that $d_u$ and $ a \in \mathcal{A} \subset \reals^{S_{{n_x},{d_u}} \times n_u}  $ is a vector of unknown coefficients. 
Under the definitions provided above, probabilistic control problem can be stated as following optimization problem.
\begin{align}\label{intro_PE1_2}
	\mathbf {P_{PC}^*} :=&\ \max_{a \in \mathcal{A}}\hbox{Prob}_{\mu_{{x_0}},\mu_{\delta},\{\mu_{\omega_k}\}_{k=0}^{N-1}}\left\lbrace x_N \in \chi_N, \{x_k \in \chi_{x_k}\}_{k=1}^{N-1}  \right\rbrace\\
	\hbox{s.t.}\quad &  x_{k+1}= f(x_k,u(x_k),\omega_k,\delta), \subeqn  \\
	& u_j(x_k) = \sum_{||i||_1\leq d_u} a_{ji} x_k^i, j=1,...,n_u  \subeqn \\
	& x_0 \sim \mu_{x_0}, \delta \sim \mu_{\delta}, \omega_k \sim \mu_{\omega_k}, k=0,\dots,N-1 \subeqn
\end{align}
Clearly, $x_k$ can explicitly be written in terms of control coefficients vector $a$, initial states $x_0$, uncertain parameters $\delta$, and disturbances $\{\omega_j\}_{j=0}^{k-1}$, using the dynamical system given in \eqref{intro_PE1} as $x_k = \mathcal{P}_{x_k} (x_0,\delta,\{\omega_i\}_{i=0}^{k-1},a)$ where $\mathcal{P}_{x_k}: \reals^{n_x + n_{\delta} + k \times n_{\omega} + S_{{n_x},{d_u}} \times n_u} \rightarrow \reals^{n_x}$ is a polynomial. Then, probabilistic control problem in \eqref{intro_PE1_2} can be restated as a volume optimization problem as in \eqref{intro_P1} as follow:
\begin{align}\label{intro_PE1_3}
	\mathbf {P_{PC}^*} := \max_{a \in \mathcal{A}} \mbox{vol}_{\mu_{{x_0}},\mu_{\delta},\{\mu_{\omega_k}\}_{k=0}^{N-1}}  \mathcal{S}_1(a)
\end{align}
where the objective function reads as
\begin{equation}
	\mbox{vol}_{\mu_{{x_0}},\mu_{\delta},\{\mu_{\omega_k}\}_{k=0}^{N-1}} \mathcal{S}_1(a) =\int_{\mathcal{S}_1(a)} d\mu_{x_0}d\mu_{\delta}d\mu_{\omega_0}...d\mu_{\omega_{N-1}} =\mbox{Prob} _{\mu_{{x_0}},\mu_{\delta},\{\mu_{\omega_k}\}_{k=0}^{N-1}} \mathcal{S}_1(a)
\end{equation}

The semialgebraic set $\mathcal{S}_1(a)$ is defined as
\begin{footnotesize}
	\begin{equation}
		\mathcal{S}_1(a) = \left\lbrace  (x_0, \delta,\{\omega_k\}_{k=0}^{N-1}): \mathcal{P}_{N}\left( \mathcal{P}_{x_N} (x_0, \delta,\{\omega_k\}_{k=0}^{N-1}, a)\right) \geq 0,\ \left\lbrace \mathcal{P}_{k}\left( \mathcal{P}_{x_k} (x_0, \delta,\{\omega_i\}_{i=0}^{k-1}, a)\right) \geq 0\right\rbrace _{k=1}^{N-1}  \right\rbrace 
	\end{equation}
\end{footnotesize}
In defined volume optimization \eqref{intro_PE1_3}, the set $\cS_2$ is $ \reals^{n_x}$.

\subsection{Generalized Sum of Squares Problem} \label{App_GSOS}
Using SOS representation Lemma \ref{sec2:lem7}, we can find a polynomial that is strictly positive on the given semialgebraic set. Many different problems in system and control can be reformulated as SOS representation problem which results in semidefinite programming problems.

In this work, we generalize the notion of SOS and introduce a new class of SOS problems where enable us to find a strictly positive polynomial on some unknown semialgebraic sets. More precisely, we define the \textit{Generalized Sum of Squares} (GSOS) problem as follow.

\textbf{Generalized Sum of Squares:} Consider polynomial $\mathcal{P}(x,a) \in \reals_d[x]$ and semialgebraic set $\mathcal{S}_1(a):= \left\lbrace  x\in \chi:\  \mathcal{P}_{1j}(x,a)\geq0, j=1,\dots ,o_1 \right\rbrace $ where $a \in \mathcal{A} \subset \reals^m$ are unknown parameters. We aim at finding parameters $a$ such that polynomial $\mathcal{P}(x,a)$ is strictly positive on the set $\mathcal{S}_1(a)$. 

This problem can be restated as a volume optimization problem in \eqref{intro_P1} by defining the set $\mathcal{S}_2(a):= \left\lbrace  x\in \chi:\  \mathcal{P}(x,a)\geq0 \right\rbrace $. As an example of GSOS, we could consider the problem of finding polynomial Lyapunov function as in section \ref{App_Lya} where we are looking for a polynomial $\epsilon ||x||_2^2-\dot{V}(x)$ to be strictly positive on the set $\left\lbrace  x\in \chi:\  0 \leq V(x) \leq r \right\rbrace $, where $\epsilon <0$ and $r>0$.


\section{Equivalent Convex Problem on Measures and Moments}\label{sec:Primal}

In this section, we first provide an equivalent infinite linear program (LP) on finite nonnegative Borel measures to
solve the constrained volume optimization problem in \eqref{intro_P1}. Then, we provide a finite dimensional semidefinite program (SDP) in moment space.
Consider the sets of volume optimization problem $\mathcal{S}_1$ and $\mathcal{S}_2$ defined in \eqref{intro_set1} and \eqref{intro_set2}. Given polynomials $\mathcal{P}_{1j}:{{\mathbb{R}}}^n\times {{\mathbb{R}}}^m\rightarrow{\mathbb{R}}$ for $j=1,2,\dots ,o_{1} $, and polynomials $\mathcal{P}_{2j}:{{\mathbb{R}}}^n\times {{\mathbb{R}}}^m\rightarrow{\mathbb{R}}$ for $ j=1,2,\dots ,o_{2} $, we define following sets as
\begin{equation}\label{K1}
	\mathcal K_{1} :=\left\{(x,a):\mathcal{P}_{1j}(x,a)\geq0, j=1,\ldots ,o_1 \right\}
\end{equation}
\begin{equation}\label{K2}
	\mathcal K_{2} :=\left\{(x,a):\mathcal{P}_{2j}(x,a)\geq0, j=1,\ldots ,o_2 \right\}
\end{equation}

\begin{assumption} \label{Assum} Assume that $\cK_1$ and $\cK_2$ satisfy Putinar's property represented in Section \ref{sec:definitions}. This implies that sets $\cK_1$ and $\cK_2$ are compact; Hence the projections of the sets $\cK_1$ and $\cK_2$ onto $x$-coordinates and onto $a$-coordinates are also compact. Also, we assume without loss of generality that $x \in \chi:=[-1,1]^n$ and $a \in \cA:=[-1,1]^m$ and the set $(\chi \times \mathcal{A}) \setminus \cK_1 = \left\lbrace   (x,a) \in \chi \times \mathcal{A} : (x,a)  \notin \cK_1 \right\rbrace $ has a nonempty interior.
\end{assumption}

\subsection{Linear Program on Measures}

As an intermediate step in the development of finite convex relaxations of the original problem in \eqref{intro_P1}, a related infinite dimensional convex problem on measures is provided as follows. Let $\mu_x$ be the given measure supported on $\chi$ defined in constrained volume problem \eqref{intro_P1}. The sets $\mathcal{K}_1$ and $\mathcal{K}_2$ are defined as \eqref{K1} and \eqref{K2} and let $\overline{\mathcal{K}_1}$ be the complement of the set $\mathcal{K}_1$. Consider the following problem on measures
\begin{align}
	\mathbf{P_{measure}^*}:=&\ \sup_{\mu ,\mu_a} \int d\mu, \label{P20}\\
	& \hbox{s.t.}\  \mu \preccurlyeq \mu_a \times \mu_x, \label{P20_1}\subeqn\\
	&\mu_a \hbox{ is a probability measure},\ \mu_a\in\cM_+(\mathcal A), \label{P20_2} \subeqn\\
	&\mu_a \times \mu_x \in \cM_+(\overline{\mathcal K_1} \cup \mathcal K_2),\ \ \mu\in\cM_+(\mathcal K_1).  \label{P20_4}\subeqn
\end{align}
In the problem \eqref{P20}, we are looking for measures $\mu$ and $\mu_a$ supported on $\mathcal{K}_1$ and $\mathcal A$, such that $\mu$ is bounded by product measure $\mu_a \times \mu_x$ supported on $\overline{\mathcal K_1} \cup \mathcal K_2$. The following theorem, shows the equivalency of the problem in  \eqref{P20} and the original volume problem in \eqref{intro_P1}.

\begin{Theorem} \label{Theo 20}
	The optimization problem in \eqref{intro_P1} is equivalent to the infinite LP in \eqref{P20} in the following sense:
	\begin{enumerate}[i)]
		\item The optimal values are the same, i.e., $\mathbf{P_{vol}^*}=\mathbf{P_{measure}^*}$.
		\item If an optimal solution to \eqref{P20} exists, call it $\mu^*_a$, then any $a^*\in supp(\mu^*_a)$ is an optimal solution to \eqref{intro_P1}.
		\item If an optimal solution to \eqref{intro_P1} exists, call it $a^*$, then $\mu_a = \delta_{a^*}$, Dirac measure at $a^*$, and $\mu = \delta_{a^*} \times \mu_x$ is an optimal solution to \eqref{P20}.
	\end{enumerate}
\end{Theorem}

\begin{proof}
	See Appendix A.
\end{proof}

Problem \eqref{P20}, requires information of the set $\mathcal K_1$ and its complement $\overline{\mathcal{K}_1}$. From a numerical implementation point of view, this results on an ill conditioned problem. To solve problem \eqref{P20}, we first need to obtain finite relaxations that provide an outer approximation of the sets of problem \eqref{P20}. The outer approximation of the sets $\mathcal K_1$ and $\overline{\mathcal{K}_1}$ intersect and thus poor performance of the solution is observed.
To solve this, we modify the provided problem on measures as follows.

We first aim at finding the approximation of the constraint of the original problem \eqref{intro_P1}. The set that approximates the constraint of volume problem includes all design parameters $a \in \mathcal{A}$ for which the set $\mathcal{S}_1(a)$ is a subset of the set $ \mathcal{S}_2(a)$. Next, we look for parameter $a$ inside the obtained set that maximizes the volume of the set $\mathcal{S}_1(a)$.

Let $\mathcal{A}_{\mathcal{F}}$ be the set of all parameters $a \in \mathcal{A}$ for which the set $\mathcal{S}_1(a)$ is a subset of the set $ \mathcal{S}_2(a)$, i.e. 
\begin{equation}\label{AF}
	\mathcal{A}_{\mathcal{F}}:=\{a\in\mathcal{A}: \mathcal{S}_1(a) \subseteq \mathcal{S}_2(a)\} 
\end{equation}

To obtain the approximation of the set $\mathcal{A}_{\mathcal{F}}$, consider the following infinite LP on continuous functions:

\begin{align}
	\mathbf{P_{\mathcal{A}_f}^*}:=&\ \inf_{f \in \mathcal{C}(a) } \int_{\mathcal{A}} f(a) d\mu_{\mathcal{A}}, \label{P21}\\
	\hbox{s.t.}\quad & f(a) \geq 1 \hbox{ on } \mathcal{K}_1 \cap \overline{\mathcal{K}_2}, \label{P21_1}\subeqn\\
	&f(a) \geq 0 \hbox{ on } \chi \times \mathcal{A}.\label{P21_2}\subeqn
\end{align}
where, $ f \in \mathcal{C}(a) $ and $\overline{\mathcal{K}_2}$ is the complement of the set $\mathcal{K}_2$.

Then, following Theorem holds true.

\begin{Theorem} \label{Theo 1}
	Let $\mu_{\cA}$ be the Lebesgue measure of the set $\cA$. Also, let $\mathcal{I}_{\overline{\mathcal{A}_\mathcal{F}}}$ be the indicator function of the set $\overline{\mathcal{A}_{\mathcal{F}}}:=\{a\in\mathcal{A}: \mathcal{S}_1(a) \not\subseteq \mathcal{S}_2(a)\}$; i.e., $\mathcal{I}_{\overline{\mathcal{A}_\mathcal{F}}}(a) = 1$ if $a \in \overline{\mathcal{A}_\mathcal{F}}$ and 0 otherwise. Then
	\begin{enumerate}[i)]
		\item There is a sequence of continuous functions $f_i(a)$ to Problem \eqref{P21} that converges to the $\mathcal{I}_{\overline{\mathcal{A}_\mathcal{F}}}$ in $L_1$-norm sense, i.e., $\hbox{lim}_{i \rightarrow \infty} \int_{\cA}  \arrowvert f_i(a) - \mathcal{I}_{\overline{\mathcal{A}_\mathcal{F}}}(a) \arrowvert da = 0$.
		\item The set $\mathcal{A}_{f_i} = \left\{ a\in\mathcal{A}:\  f_i(a) < 1 \right\}$ converges to the set $\mathcal{A}_{\mathcal{F}}$, i.e., $\hbox{lim}_{i \rightarrow \infty} \mu_{\cA} (\cA_{\cF} - \mathcal{A}_{f_i} ) = 0$, 
		and  $\mathcal{A}_{f_i} \subseteq \mathcal{A}_\mathcal{F}$.
	\end{enumerate}
\end{Theorem}

\begin{proof}
	See Appendix B.
\end{proof}
Now, to obtain an approximate of the solution of the original volume problem \eqref{intro_P1}, consider infinite LP on measures as follows. Let, $\mu_x$ be the given measure supported on $\chi$ as in problem \eqref{intro_P1} and $\mathcal{K}_1$ be the set as in \eqref{K1}. Then, consider the following problem
\begin{align}
	\mathbf{P_{f_i}^*}:=&\ \sup_{\mu ,\mu_a} \int d\mu, \label{P2}\\
	\hbox{s.t.}\quad & \mu \preccurlyeq \mu_a \times \mu_x, \label{P2_1}\subeqn\\
	&\mu_a \hbox{ is a probability measure} \label{P2_2}, \subeqn\\
	&\mathcal{A}_{f_i} = \left\{ a\in\mathcal{A}:\  f_i(a) < 1 \right\}, \label{P2_3}\subeqn\\
	&\mu_a\in \cM_+(\mathcal{A}_{f_i}),\quad \mu\in\cM_+(\mathcal K_1).  \label{P2_4}\subeqn
\end{align}
Now, following theorem establishes the equivalence of volume optimization problem in \eqref{intro_P1} and infinite LP in \eqref{P2}.

\begin{Theorem} \label{Theo 2}
	Let, $(\mu^*_a(f_i),\mu^*(f_i),\mathbf{P_{f_i}^*})$ be an optimal solution and value of the LP in \eqref{P2} for the obtained function $f_i$ and the set $\cA_{f_i}$ solving Problem \eqref{P21}. Also, assume that volume problem \eqref{intro_P1} has a unique optimal solution and value $(a^*,\mathbf{P_{vol}^*})$. As the set $\mathcal{A}_{f_i} = \left\{ a\in\mathcal{A}:\  f_i(a) < 1 \right\}$ defined in Theorem \eqref{Theo 1} converges to the set $\mathcal{A}_{\mathcal{F}}$, we have the following results:
	\begin{enumerate}[i)]
		\item The optimal value $\mathbf{P_{f_i}^*}$ converges to the $\mathbf{P_{vol}^*}$.
		\item Measures $\mu^*_a(f_i)$ and $\mu^*(f_i)$ converge to $\mu_a = \delta_{a^*}$, Dirac measure at $a^*$, and $\mu = \delta_{a^*} \times \mu_x$, respectively.
		\item Any point in the support of the measure $\mu^*_a(f_i)$, i.e., $a_i \in supp(\mu^*_a(f_i))$, converges to the $a^*$.
	\end{enumerate}
\end{Theorem}
\begin{proof}
	See Appendix C.
\end{proof}


In the next section, we provide the tractable finite relaxation to infinite LP in \eqref{P2} and \eqref{P21}.
\subsection{Finite Semidefinite Program on Moments}

In this section, we provide a finite dimensional SDP whose feasible region is defined over real sequences. We show that the corresponding sequence of optimal solutions arbitrarily approximate the optimal solution of \eqref{P2}.
Unlike problem \eqref{P2} in which we are looking for measures, in the provided SDP formulation, we aim at finding moment sequences corresponding to measures that are optimal to \eqref{P2}. 

For this purpose, we first need to obtain the semialgebraic approximation of the set $\cA_{\cF}$ in \eqref{AF}, the set of all parameters $a\in \cA$ for which the set $\cS_1(a)$ is a subset of the set $\cS_2(a)$. In the previous section, the continuous function $f$ and the infinite LP in \eqref{P21} are used to obtain an inner approximation of the set  $\cA_{\cF}$. Here, we use polynomial $ \mathcal{P}^d_{\mathcal{A}}(a) \in \reals_d[a] $ and finite finite SDP problem below
\begin{align}
	\mathbf{P_{\mathcal{A}_{d}}^*}:=&\ \min_{\mathcal{P}^d_{\mathcal{A}}(a) \in \reals_d[a]} \int_{\mathcal{A}} \mathcal{P}^d_{\mathcal{A}}(a) d\mu_{\mathcal{A}}, \label{P21d}\\
	\ \hbox{s.t.} \  \mathcal{P}^d_{\mathcal{A}}(a)& - 1 \in \mathcal{QM}_i \left( \{\mathcal{P}_{1j}\}_{j=1}^{o_1}, -\mathcal{P}_{2i} \right), i=1,\dots ,o_2  \subeqn\label{P21d_1}\\
	\mathcal{P}^d_{\mathcal{A}}(a) & \in \mathcal{QM} \left( \{(1-x^2_i)\}_{i=1}^n, \{(1-a^2_i)\}_{i=1}^m \right).\label{P21d_2}\subeqn
\end{align}
where, $\mu_{\mathcal{A}}$ is the Lebesgue measure over the set $\mathcal{A}$ and $d$ is the order of polynomial $ \mathcal{P}^d_{\mathcal{A}}(a)$. $\mathcal{QM}_i$ and $\mathcal{QM}$ as defined in \eqref{sec1:def1} are the quadratic modules generated by polynomials of set $\cK_1 \cap \overline{\cK_2}$ \begin{footnotesize} $ =\left\lbrace (x,a): \cup_{i=1}^{o_2}\{ -\mathcal{P}_{2i} > 0 , \mathcal{P}_{1j}\geq 0, j=1,\dots ,o_1 \}  \right\rbrace, $
\end{footnotesize} and polynomials of hyper cube $\chi \times \cA$, respectively. According to the Lemma \ref{sec2:lem7}, constraints \eqref{P21d_1} and \eqref{P21d_2} imply that polynomials $(\mathcal{P}^d_{\mathcal{A}}(a) - 1)$ and  $\mathcal{P}^d_{\mathcal{A}}(a)$ are positive on the sets $\cK_1 \cap \overline{\cK_2}$ and hyper cube $\chi \times \cA$, respectively. Problem in \eqref{P21d} is a SDP, where objective function is a weighted summation of coefficients of polynomial $\mathcal{P}^d_{\mathcal{A}}(a)$ with respect to the moments of Lebesgue measure $\mu_{\mathcal{A}}$ and constraints are convex linear matrix inequalities in terms of coefficients of polynomial $\mathcal{P}^d_{\mathcal{A}}(a)$. 

The following theorem hold true for the problems \eqref{P21d}.

\begin{Theorem} \label{Theo 3}
	Let $\mathcal{P}^d_{\mathcal{A}}(a)$ be an optimal solution of SDP \eqref{P21d} and consider the set:
	\begin{equation}\label{Ad}
		\mathcal{A}_{{d}} = \left\{ a\in\mathcal{A}:\  \mathcal{P}^d_{\mathcal{A}}(a) < 1 \right\}
	\end{equation}
	Also, let $\mathcal{I}_{\overline{\mathcal{A}_\mathcal{F}}}$ be the indicator function of the set $\overline{\mathcal{A}_{\mathcal{F}}}:=\{a\in\mathcal{A}: \mathcal{S}_1(a) \not\subseteq \mathcal{S}_2(a)\}$. Then,
	
	\begin{enumerate}[i)]
		\item	The sequence of optimal solutions to the finite SDP in \eqref{P21d} converges to the $\mathcal{I}_{\overline{\mathcal{A}_\mathcal{F}}}$ in $L_1$-norm sense as $d \rightarrow \infty$, i.e., $\hbox{lim}_{d \rightarrow \infty} \int_{\cA}  \arrowvert \cP^d_{\cA}(a) - \mathcal{I}_{\overline{\mathcal{A}_\mathcal{F}}}(a) \arrowvert da = 0$.
		
		\item The set $	\mathcal{A}_{{d}}$ converges to the set $\mathcal{A}_{\mathcal{F}}$ in \eqref{AF}, i.e., $\hbox{lim}_{d \rightarrow \infty} \mu_{\cA} (\cA_{\cF} - \mathcal{A}_{d} ) = 0$, 
		and  $\mathcal{A}_{d} \subseteq \mathcal{A}_\mathcal{F}$.
	\end{enumerate}	
	
\end{Theorem}

\begin{proof}
	See Appendix D.
\end{proof}


Given that the indicator function $\mathcal{I}_{\overline{\mathcal{A}_\mathcal{F}}}$ can be approximated by sequence of polynomials of increasing order $\cP^d_{\cA}$ as in Theorems \ref{Theo 3}, we restrict the continuous function $f_i$ in problem \eqref{P2} to be polynomials. Hence, we can approximate the optimal solution of infinite LP \eqref{P2} on measures with finite dimensional SDP on moments. In order to have tractable approximations to the infinite dimensional LP in \eqref{P2}, we consider the SDP \eqref{P2M}, known as Lasserre's hierarchy \cite{Lasserre}, where $\mathbf{y_x}:=\{y_{x_\beta}\}_{\beta\in\mathbb{N}_{2r}^n}$ and $\mathbf{y_a}=\{y_{a_\alpha}\}_{\alpha\in\mathbb{N}_{2r}^m}$ are the truncated moment sequence of measures $\mu_x$ and $\mu_a$.
\begin{align}
	\mathbf{P^*_r}:= &\sup_{\mathbf{y}\in\reals^{S_{n+m,2r}},\ \mathbf{y_a}\in\reals^{S_{m,2r}}} (\mathbf{y})_\mathbf{0},
	\label{P2M}\\
	\hbox{s.t.}\quad & M_r(\mathbf y)\succcurlyeq 0,\ M_{r-r_j}(\mathbf{y}; \mathcal{P}_{1j})\succcurlyeq 0,\ j=1,\dots ,o_1, \label{P2M_1}\subeqn\\
	& \left(\mathbf{y_a}\right)_\mathbf{0}=1, \label{P2M_2}\subeqn\\
	&M_r ({\mathbf y}_{\mathbf a})\succcurlyeq 0,\ M_{r-r_a}(\mathbf{y}_{\mathbf a}; 1-\mathcal{P}^d_{\mathcal{A}}(a))\succcurlyeq 0, \ M_{r-1}(\mathbf{y}_{\mathbf a}; 1-a^2_i )\succcurlyeq 0, i=1,...,m \label{P2M_3}\subeqn\\
	&M_r (\mathbf{y_a}\times\mathbf{y_x}-{\mathbf y})\succcurlyeq 0.\label{P2M_4}\subeqn
\end{align}
In \eqref{P2M}, $(\mathbf{y})_\mathbf{0}$ is the first element of the truncated moment sequence of measure $\mu$, $r\in\integers_+$ is relaxation order of matrices, $d_j$ is the degree of polynomial $\cP_{1j}$ in the set $\mathcal{S}_1$, $r_j:=\left\lceil\frac{d_j}{2}\right\rceil$ for all $1\leq j\leq o_1$. Sequence $\mathbf{y_a}\times\mathbf{y_x}=\mathbf{\bar{y}}$ is truncated moment sequence of measure $\mu_a \times \mu_x$ such that $(\mathbf{\bar{y}})_\theta=(\mathbf{y_a})_\alpha (\mathbf{y_x})_\beta$ for all $\theta=(\alpha,\beta)\in\mathbb{N}_{2r}^{n+m}$. Matrices $M_r ({\mathbf y})$, $M_r ({\mathbf y}_{\mathbf a})$, and $M_r (\mathbf{y_a}\times\mathbf{y_x}-{\mathbf y})$ are moment matrices constructed by moment sequences $\mathbf y$, ${\mathbf y}_{\mathbf a}$, and $\mathbf{y_a}\times\mathbf{y_x}-{\mathbf y}$, respectively. Also, $M_{r-r_j}(\mathbf{y}; \mathcal{P}_{1j}), j=1,\dots ,o_1$ and $\left\lbrace \ M_{r-1}(\mathbf{y}_{\mathbf a}; 1-a^2_i )\right\rbrace_{i=1}^m$ are localization matrices constructed by polynomials of the set $\mathcal{K}_1$ and hyper cube  $\chi \times \cA$, respectively. Finally, $M_{r-r_a}(\mathbf{y}_{\mathbf a}; 1-\mathcal{P}^d_{\mathcal{A}}(a))$ is localization matrix constructed by polynomial of the set $\cA_d$ in \eqref{Ad}, i.e., $(1-\mathcal{P}^d_{\mathcal{A}}(a))$, where $\mathcal{P}^d_{\mathcal{A}}(a)$ is an optimal solution of SDP \eqref{P21d}.



\begin{remark} \label{remark2}
	To be able to work with closed sets $\mathcal{A}_{{d}}$ in \eqref{Ad} and $\mathcal{K}_1 \cap \overline{\mathcal{K}_2}$ which are used in constraints \eqref{P2M_3} and \eqref{P21d_1}, we use positive small $\epsilon_{\mathcal{A}}, \epsilon_{\mathcal{K}} \rightarrow 0$ and also to satisfy the Putinar's property, we add the polynomial $ \sqrt{m}^2- \Vert a \Vert^2 \geq 0$, i.e., $$\mathcal{A}_{{d}} = \left\{ a\in\mathcal{A}:\  \mathcal{P}^d_{\mathcal{A}}(a) \leq 1-\epsilon_{\mathcal{A}}, \sqrt{m}^2- \Vert a \Vert^2 \geq 0 \right\}$$ and 
	$$\mathcal{K}_1 \cap \overline{\mathcal{K}_2}=\left\lbrace (x,a): \cup_{i=1}^{o_2}\{ -\mathcal{P}_{2i} \geq \epsilon_{\mathcal{K}} , \mathcal{P}_{1j}\geq 0, j=1,\dots ,o_1 \}  \right\rbrace $$
\end{remark}

\subsection{Illustrative Example}\label{SExa1}

In this section, we present a simple example of constrained volume optimization problem in \eqref{intro_P1} and show how the proposed finite SDP in \eqref{P2M} effectively works. For illustrative purposes, the provided example is low dimensional and consists of sets described by polynomials in $x \in \chi \subset \reals$ and parameter $a \in \mathcal{A} \subset \reals$. We consider the volume optimization problem in \eqref{intro_P1} with following sets
\begin{align} \label{E1_set1}
	\mathcal{S}_1(a):= \left\lbrace  x\in \chi:\  0.25 - a^2 -x^2 \geq0 \right\rbrace 
\end{align}
\begin{align} \label{E1_set2}
	\mathcal{S}_2(a):= \left\lbrace  x\in \chi:\  0.09 - a^2 - 0.8a - x^2 \geq0  \right\rbrace 
\end{align}
where, $\chi=[-1, 1]$ and given measure $\mu_x$ is the Lebesgue measure supported on $\mathcal{A}=[-1, 1]$. 

\begin{figure}[!h]
	\centering
	\includegraphics[scale=0.26]{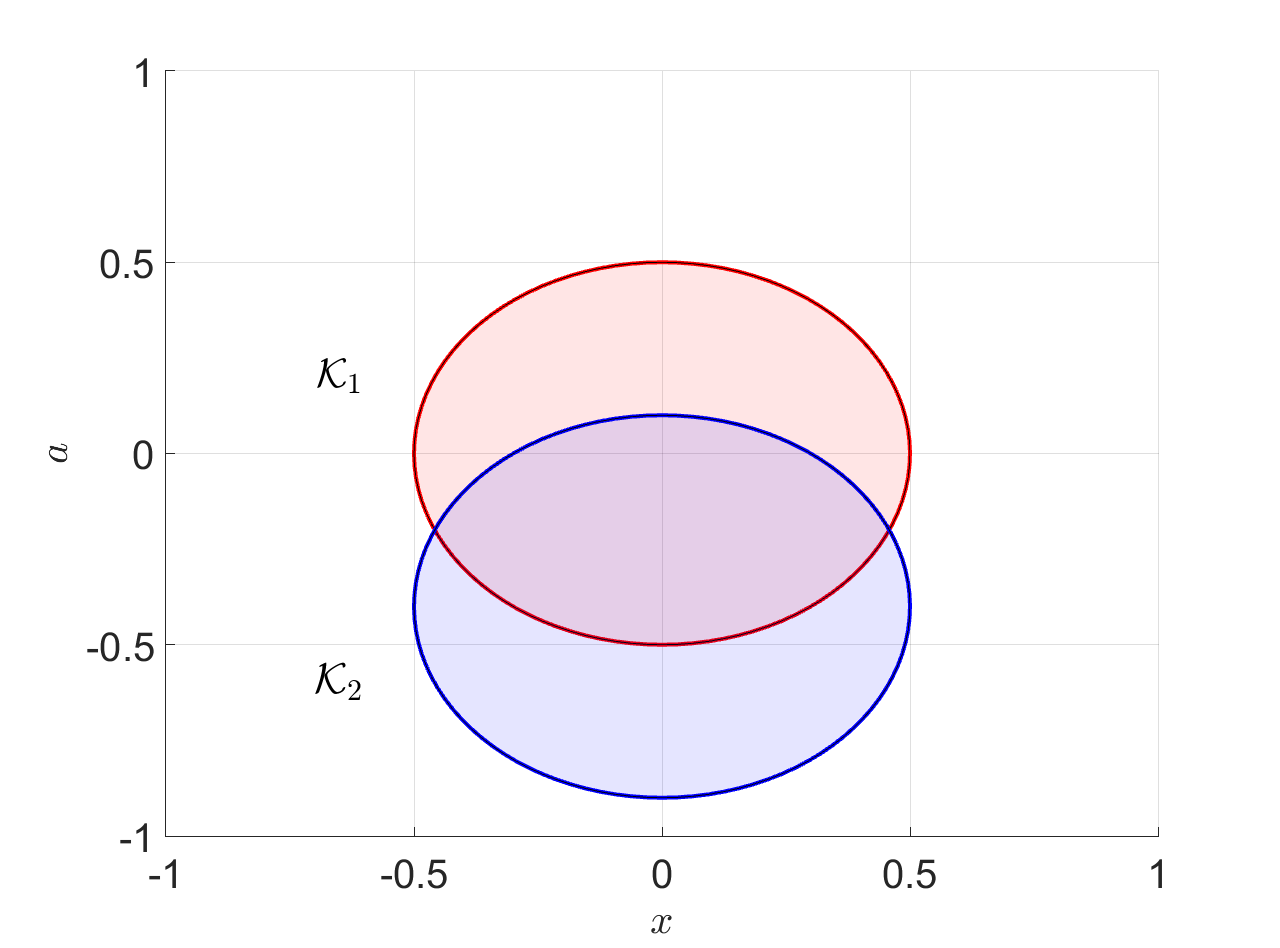}
	\caption{Sets $\cK_1$ and $\cK_2$}
	\label{fig:1} 
\end{figure}

Based on given sets $\cS_1$ and $\cS_2$, we define the sets {\small $\mathcal K_{1} :=\left\{(x,a):0.25 - a^2 -x^2\geq0 \right\}$ } and {\small $\mathcal K_{2} :=\left\{(x,a):0.09 - a^2 - 0.8a - x^2 \geq0 \right\}$}. Figure \ref{fig:1} displays the sets $\cK_1$ and $\cK_2$. To obtain an approximate solution of constrained volume optimization problem, we solve finite SDPs in \eqref{P2M} and \eqref{P21d}. First, we solve the SDP in \eqref{P21d} to obtain the polynomial $\cP^d_{\cA}(a)$. To this, we use Yalmip which is a MATLAB-based toolbox for polynomial and SOS optimization \cite{Efberg2004}. Figure \ref{fig:2} displays polynomial $\cP^d_{\cA}(a)$ obtained by SDP \eqref{P21d} for polynomial order $d=7$. As constraints of SDP \eqref{P21d}, $\cP^d_{\cA}(a)$ is greater than 1 on $\mathcal K_{1} \cap \overline{\mathcal K_{2}}$ 
and is positive on $\chi \times \cA = [-1, 1]^2$. Hence, based on Theorem \ref{Theo 3} the set {\small $\mathcal{A}_{d} = \left\{ a\in \reals:\  \mathcal{P}^7_{\mathcal{A}}(a) \leq 1-\epsilon_{\cA}, 1-a^2 \geq 0, \epsilon_{\cA} = 0.05 \right\}$} is an inner approximation of the set $\cA_{\cF}$, the set of all parameter $a \in \cA$ that set $\cS_1$ is subset of $\cS_2$. Clearly, based on the Figure \ref{fig:1}, $\cA_{\cF} = (-1 \leq a \leq -0.2) \cup (0.5 \leq a \leq 1)$, where $\cS_1(a) \subseteq \cS_2(a) $ for $a \in [-0.5 ,-0.2]$, $\cS_1(a) = \varnothing \subseteq \cS_2(a) $ for $a \in [-0.9,-0.5)$, and $\cS_1(a) = \cS_2(a)=\varnothing $ for $a \in [-1,-0.9)\cup (0.5,1]$. Figure \ref{figN:3} displays polynomial $\cP^d_{\cA}(a)$ obtained by SDP \eqref{P21d} for different polynomial orders $d=2,4,6,7$ and also the sets $\cA_{d}$.

\begin{figure}
	\centering
	\includegraphics[scale=0.26]{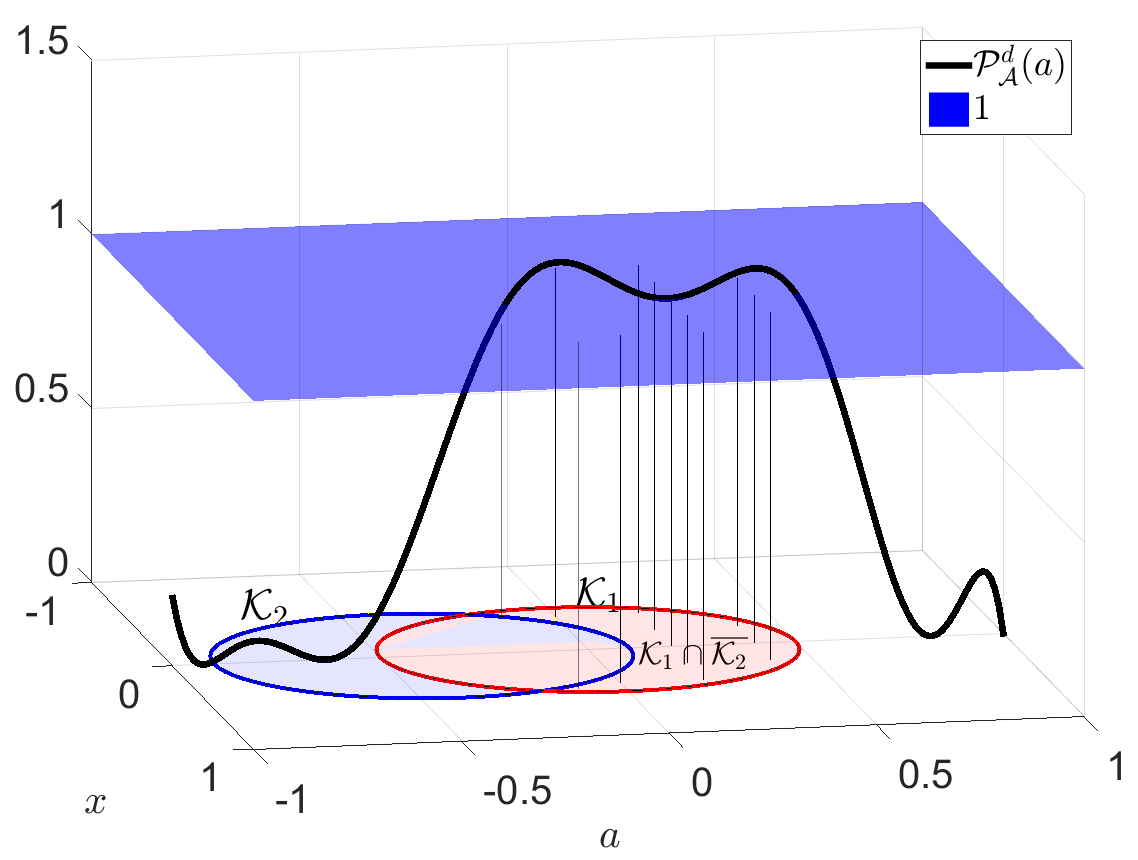}
	\caption{Polynomial $\cP^d_{\cA}(a)$ obtained by SDP \eqref{P21d} for $d=7$  }
	\label{fig:2} 
\end{figure}

\begin{figure}[!h]
	\centering
	\includegraphics[scale=0.26]{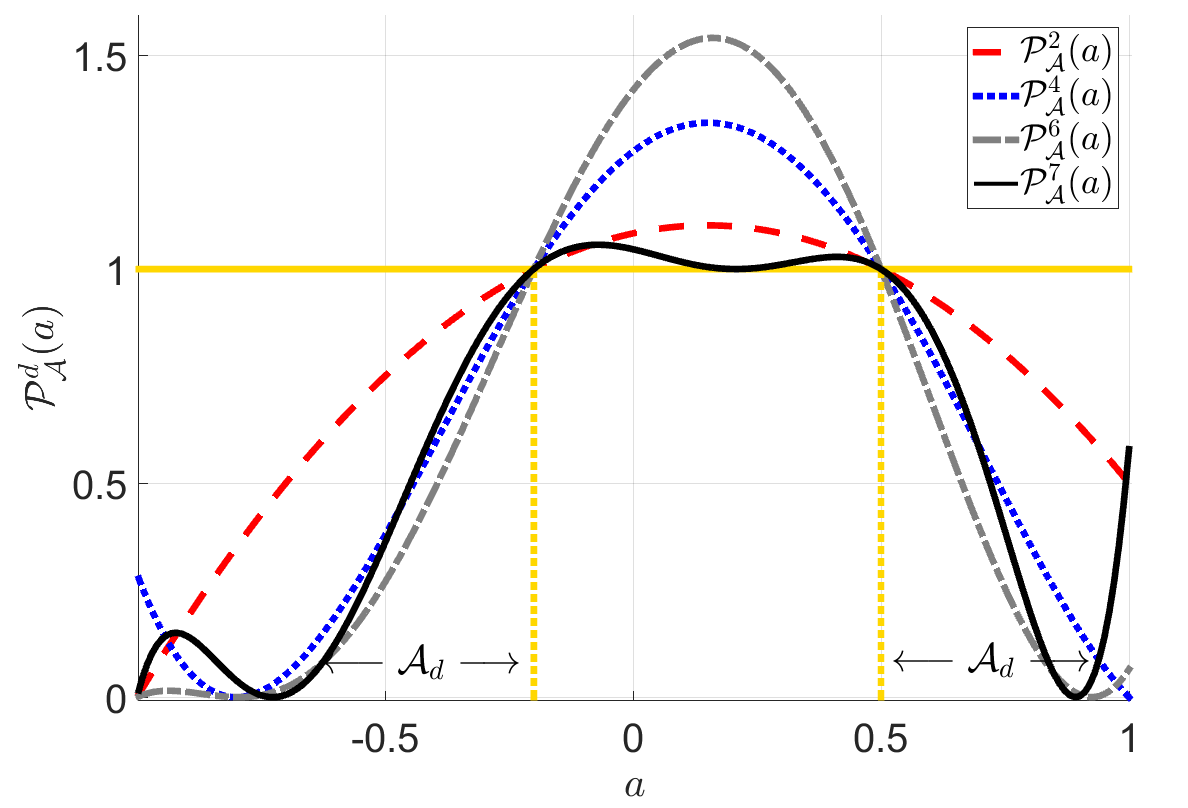}
	\caption{Polynomial $\cP^d_{\cA}(a)$ obtained by SDP \eqref{P21d} for $d=2,4,6,7$  }
	\label{figN:3} 
\end{figure}

We take $\cP^7_{\cA}(a)$ and solve SDP in \eqref{P2M}. Based on moments of Lebesgue measure $\mu_x$ on $\mathcal{A}=[-1, 1]$ as $(y_x)_{\alpha} = \frac{1}{(\alpha+1)}(1^{\alpha+1} - (-1)^{\alpha+1})$, we construct the matrices in constraints of SDP \eqref{P2M} in terms of unknown moment vectors $\mathbf{y}\in\reals^{S_{2,2r}}$ and $ \mathbf{y_a}\in\reals^{S_{1,2r}}$. Since the order of maximum degree of polynomials in $\cS_1$ and $\cA_d$ is 7, the minimum relaxation order for SDP \eqref{P2M} is $r=4$, which requires the moments up to order 8. The SDP in \eqref{P2M} with $r=6$ is solved using GloptiPoly. Based on obtained solution for moment vectors, we approximate the solution to volume problem $a$ with ${y_a}_1 = -0.2050$ and estimate the optimal volume $\mathbf{P_{vol}^*}$ with $\mathbf{P_r} = y_{00} = 1.239$. Clearly, for obtained $a =-0.2050$, the set $\cS_1 = \{ x :    0.4561^2 - x^2 \geq 0\}$ is a subset of $\cS_2 = \{ x :  0.4604^2 - x^2 \geq 0\}$.  Based on Figure \ref{fig:1}, the true solution for the volume optimization problem is $a^*=-0.2$ with volume $\mathbf{P_{vol}^*} = 0.9165$. To obtain better estimates of the optimum volume, one needs to increase the relaxation order r. Also, see (\cite{Jasour2015}, section 3.3), where we provided some methods to improve the estimated volume of the semialgebraic sets in the similar setup.

\section{Dual Convex Problem on Function Space}\label{sec:Dual}

In this section, we provide an infinite LP on continuous functions which is dual to the infinite LP on measure in \eqref{P2}.
The provided dual problem gives a new insight on solving the volume problem. Also, from computational efficiency perspective 
we can take advantage of polynomial convex optimization techniques. For instance, to handle large scale SDPs, one can employ DSOS optimization technique where relies on linear and second order cone programming (\cite{Ahmadi2014,Ahmadi2016}).

To obtain a dual problem to the infinite LP in \eqref{P2}, let $\mathcal{C(\chi \times \mathcal{A})}$ be the Banach space of continuous functions on $\chi \times \mathcal{A}$. Then, Lagrangian dual of \eqref{P2} is:
\begin{align}\label{PD}
	\mathbf{P_{Dual}^*}:=&\ \inf_{\beta \in \reals,\mathcal{W} \in \mathcal{C(\chi \times \mathcal{A})}} \beta, \\
	&\hbox{s.t.}\quad \mathcal{W}(x,a) \geq 1 \quad \hbox{on} \quad \mathcal{K}_1, \label{PD_1}\subeqn\\
	&\beta - \int_\chi \mathcal{W}(x,a) d\mu_x \geq 0 \quad \hbox{on} \quad \mathcal{A}_{f_i}, \label{PD_2}\subeqn \\
	&\mathcal{W}(x,a) \ge 0, \ \beta \geq 0. \subeqn \label{PD_3} 
\end{align}
where, $\mathcal{K}_1$ is defined as \eqref{K1}, $\mu_x$ is a given Borel measure, and $\mathcal{A}_{f_i}$ is a set defined in \eqref{P2_3}. We can interpret the obtained dual problem as follow. If we assume that $a$ is given, then the optimal solution for $\mathcal{W}(x,a)$ is the indicator function of the set $\cK_1$ and the optimal value $\mathbf{P_{Dual}^*}$ is the volume of the set $\cK_1$, i.e., $\mathbf{P_{Dual}^*} = \beta = \int_\chi \mathcal{W}(x,a) d\mu_x$. Otherwise, $\int_\chi \mathcal{W}(x,a)d\mu_x$ is an upper bound for the volume of the set $\cK_1$.

The following theorem establish the equivalence of problems in \eqref{P2} and \eqref{PD}.
\begin{Theorem}\label{Theo 5}
	There is no duality gap between the infinite LP on measure in \eqref{P2} and infinite LP on continuous function in \eqref{PD} in the sense that the optimal values are the same, i.e., $\mathbf{P^*_{f_i}} = \mathbf{P^*_{Dual}}$
\end{Theorem}

\begin{proof}
	See Appendix E.
\end{proof}

To be able to obatin a tractable relaxation of infinite LP in \eqref{PD}, we use polynomial approximation of continuous function $\mathcal{W}$ and use SOS relaxation to satisfy the nonnegativity constraints, where results in following finite SDP on polynomials:
\begin{align}\label{PD2F}
	\mathbf{P_d^*}:=&\ \min_{\beta \in \reals,\mathcal{P}^d_{\mathcal{W}} \in \reals_d[x,a]} \beta, \\
	&\hbox{s.t.}\quad \mathcal{P}^d_{\mathcal{W}}(x,a)-1 \in \mathcal{QM} \left( \{\mathcal{P}_{1j}\}_{j=1}^{o_1} \right), \label{PD2F_1}\subeqn\\
	&\beta - \int_\chi \mathcal{P}^d_{\mathcal{W}}(x,a) d\mu_x \in \mathcal{QM} \left(\{1-\mathcal{P}^d_{\mathcal{A}}(a)\},\{(1-a^2_i)\}_{i=1}^m \right), \label{PD2F_2}\subeqn \\
	&\mathcal{P}^d_{\mathcal{W}}(x,a) \geq 0, \ \beta \geq 0. \subeqn \label{PD2F_3} 
\end{align}
where, $\mathcal{P}^d_{\mathcal{W}}(x,a) \in \reals_d[x,a]$, $\mu_x$ is a given finite Borel measure and $\cQ \cM$ defined in \eqref{sec1:def1} is quadratic module generated by polynomials. According to the Lemma \ref{sec2:lem7}, constraints \eqref{PD2F_1} and \eqref{PD2F_2} imply that polynomials $\mathcal{P}^d_{\mathcal{W}}(x,a)-1$ and $\beta - \int_\chi \mathcal{P}^d_{\mathcal{W}}(x,a) d\mu_x$ are positive on the sets $\cK_1$ in \eqref{K1} and $\cA_d=\{ a \in \cA : 1-\mathcal{P}^d_{\mathcal{A}}(a) > 0 \}$ in \eqref{Ad}, respectively, where $\mathcal{P}^d_{\mathcal{A}}(a)$ is an optimal solution of SDP \eqref{P21d}. Problem in \eqref{PD2F} is a SDP, where objective function is a linear and constraints are convex linear matrix inequalities in terms of coefficients of polynomial $\mathcal{P}^d_{\mathcal{W}}$. To be able to work with closed set $\mathcal{A}_{{d}}$, see the Remark \ref{remark2}.

The following theorem establish the equivalence of problems in \eqref{P2M} and \eqref{PD2F}.

\begin{Theorem}\label{Theo 6}
	There is no duality gap between the finite SDP on moments in \eqref{P2M} and finite SDP on polynomials in \eqref{PD2F} in the sense that the optimal values are the same, i.e., $\mathbf{P^*_r} = \mathbf{P^*_d}.$
\end{Theorem}

\begin{proof}
	See Appendix F. 
\end{proof}

\begin{figure}[!h]
	\centering
	\includegraphics[scale=0.26]{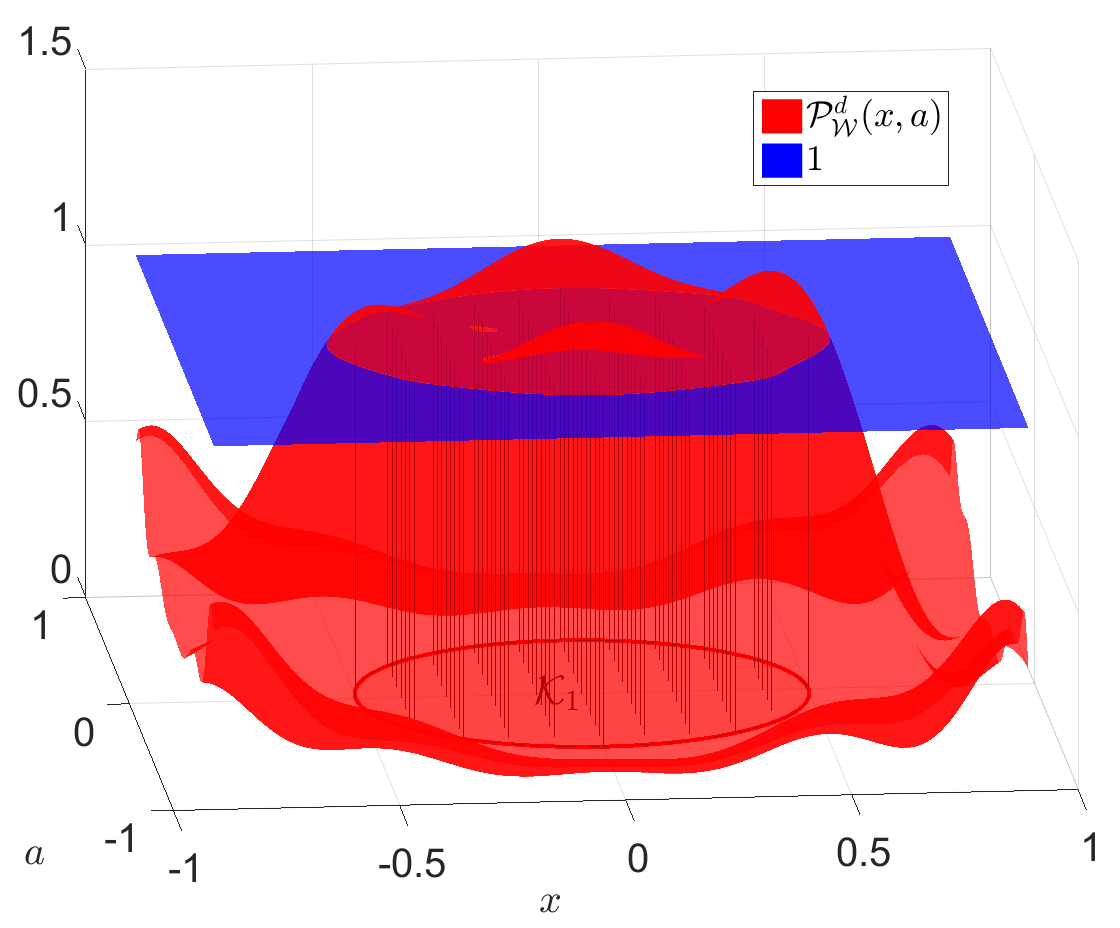}
	\caption{Polynomial $\cP^d_{\cW}(x,a)$ obtained by SDP \eqref{PD2F} for $d=12$  }
	\label{fig:3} 
\end{figure}

\begin{remark} \label{remark3}
	In low dimensional problems, we can replace the global positivity condition in \eqref{PD2F_3} with local constraint as $\{ \mathcal{P}^d_{\mathcal{W}}(x,a) \geq 0 \ \ \hbox{on} \ \  \chi \times \cA \}$ to improve the obtained results.
\end{remark}

\subsection{Illustrative Example}
Consider the simple example provided in section \ref{SExa1}. Here, to obtain an approximate solution, we solve the dual problem provided in finite SDP \eqref{PD2F}. We take $\cP^7_{\cA}(a)$ obtained by solving \eqref{P21d} and solve SDP in \eqref{PD2F} for polynomial order $d=12$ by Yalmip. Figure \ref{fig:3} displays obtained $\cP^d_{\cW}(x,a)$ which is greater than 1 on the set $\cK_1$ and is positive on $\chi \times \cA = [-1, 1]^2$ as in constraint \eqref{PD2F_1}. Figure \ref{fig:4} displays obtained $\beta$ and also {\small $\int_{\chi} \cP^d_{\cW}(x,a)d\mu_x$}. As in constraint \eqref{PD2F_2} $\beta$ is greater than {\small $\int_{\chi} \cP^d_{\cW}(x,a)d\mu_x$} on the set $\cA_d=\{ a\in \cA: \cP^7_{\cA}(a) < 1 \}$. Based on obtained $\beta$ and {\small $\cP^{12}_{\cW}(x,a)$}, we approximate the solution to the volume optimization problem with $a = -0.2050$ that maximizes polynomial  $\int_{\chi} \cP^d_{\cW}(x,a)d\mu_x$ on the $\cA_d$ and estimate the optimal volume $\mathbf{P_{vol}^*}$ with $\mathbf{P_d} = \beta = 1.239$. Based on the Theorem \ref{Theo 6}, the obtained solution by solving dual SDP in \eqref{PD2F} matches the solution obtained by SDP in \eqref{P2M}.

\begin{figure}[!h]
	\centering
	\includegraphics[scale=0.26]{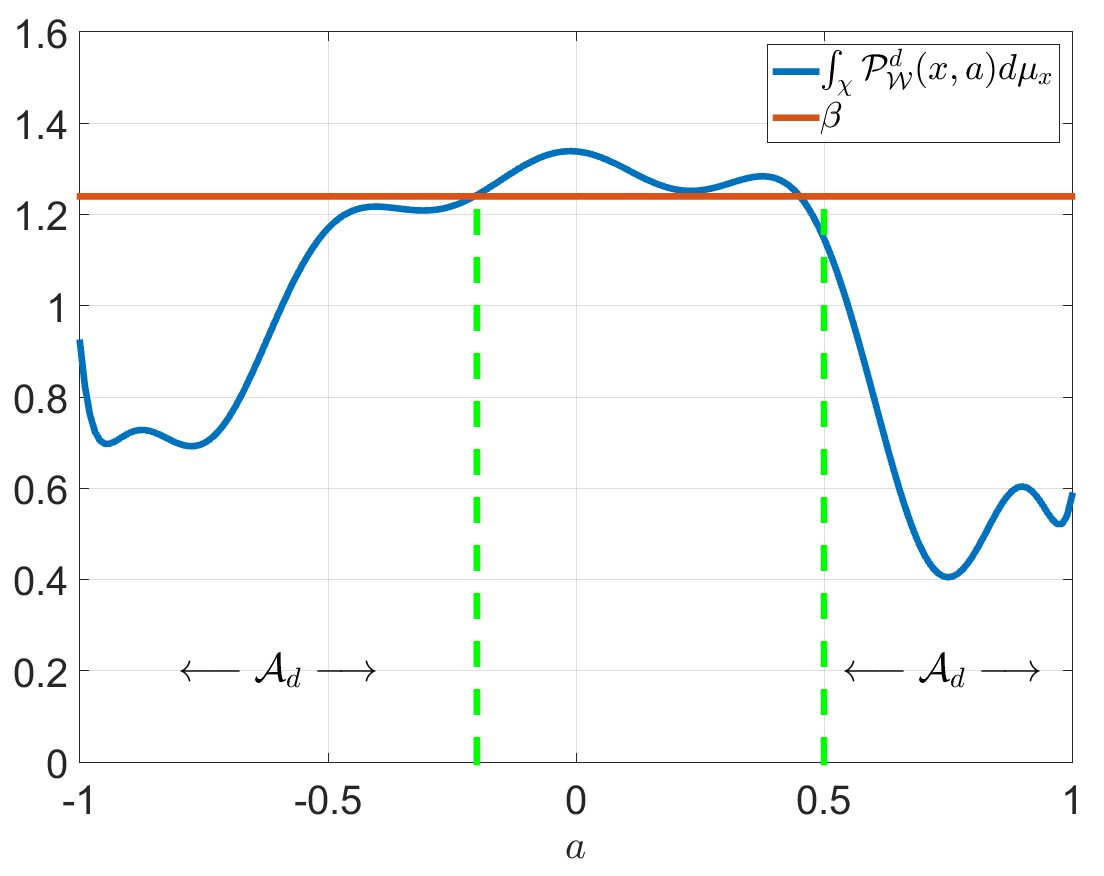}
	\caption{$\beta$ and $\int_{\chi} \cP^d_{\cW}(x,a)d\mu_x$ obtained by SDP \eqref{PD2F} for $d=12$  }
	\label{fig:4} 
\end{figure}

\section{Implementation and Numerical Results}\label{sec:Res}

In this section, numerical examples are presented that illustrate the performance of proposed method. The presented example are problem of inner approximation of ROA set defined in section \ref{App_Lya} and problem of probabilistic control defined in \ref{App_PC}.

\subsection{Example 1: ROA set of system}\label{Exa1}
In this example, we address the problem of approximating ROA set defined in \ref{App_Lya}. Consider the following locally stable nonlinear system.

\begin{equation}\label{E1}
	\begin{aligned}
		\dot{x}_1 =& -x_2\\ 
		\dot{x}_2 =&\ x_1 + (4x_1^2-1)x_2 
	\end{aligned}
\end{equation}
where, states of the system $x \in \chi = [-1, 1]^2 $. To approximate the ROA set of the system in the unit box, the Lyapunov function is described as 
\begin{equation}
	V(x) = 3\lVert x  \lVert_2^2 + 3a_1x_1x_2 + 3a_2x_1^3x_2 + 3a_3x_1x_2^3
\end{equation}
where $a=[a_1, a_2, a_3] \in \mathcal{A}=[-1, 1]^3 $ is the vector of unknown coefficients. The equivalent constrained volume optimization problem is stated as \eqref{Lya}. To obtain an approximate solution, we solve finite SDPs in \eqref{P2M} and \eqref{P21d}. First, we solve the SDP in \eqref{P21d} to obtain the polynomial $\cP^d_{\cA}(a)$ for $d=10$. The polynomials describing the sets $\cK_1$ and $\cK_2$ are: 
\begin{equation}
	\mathcal{P}_{11} = V(x,a), \ \ \mathcal{P}_{12} = 1-V(x,a)
\end{equation}

\begin{equation}
	\mathcal{P}_{21} = -\epsilon_r\Vert x \Vert_2^2 -\dfrac{\partial V(x,a)}{\partial x_1}\dot{x}_1- \dfrac{\partial V(x,a)}{\partial x_1} \dot{x}_2
\end{equation}
We set $\epsilon_r$ to 0.001 and $\epsilon_{\mathcal{K}}$ and $\epsilon_{\mathcal{A}}$ as in Remark \ref{remark2} to 0.1 and 0.02, respectively. 
Figure \ref{fig:6} shows the obtained set {\small $ \left\{ (a_1,a_2,a_3):\  \mathcal{P}^{10}_{\mathcal{A}}(a) \leq 1-\epsilon_{\mathcal{A}} \right\}$}. Based on Theorem \ref{Theo 3}, the set  {\small $\mathcal{A}_{d} = \left\{ (a_1,a_2,a_3):\  \mathcal{P}^{10}_{\mathcal{A}}(a) \leq 1-\epsilon_{\mathcal{A}}, \{1-a_i^2 \geq 0\}_{i=1}^3 \right\}$} is an inner approximation of the set of all coefficients $(a_1,a_2,a_3)$ for which the set $\{ x\in \chi :\ 0 \leq V(x,a) \leq 1 \} $ is subset of  the set $ \{x\in \chi : \dot{V}(x,a) \leq - \epsilon_r \Vert x \Vert_2^2 \} $.

\begin{figure}[!h]
	\centering
	\includegraphics[scale=0.26]{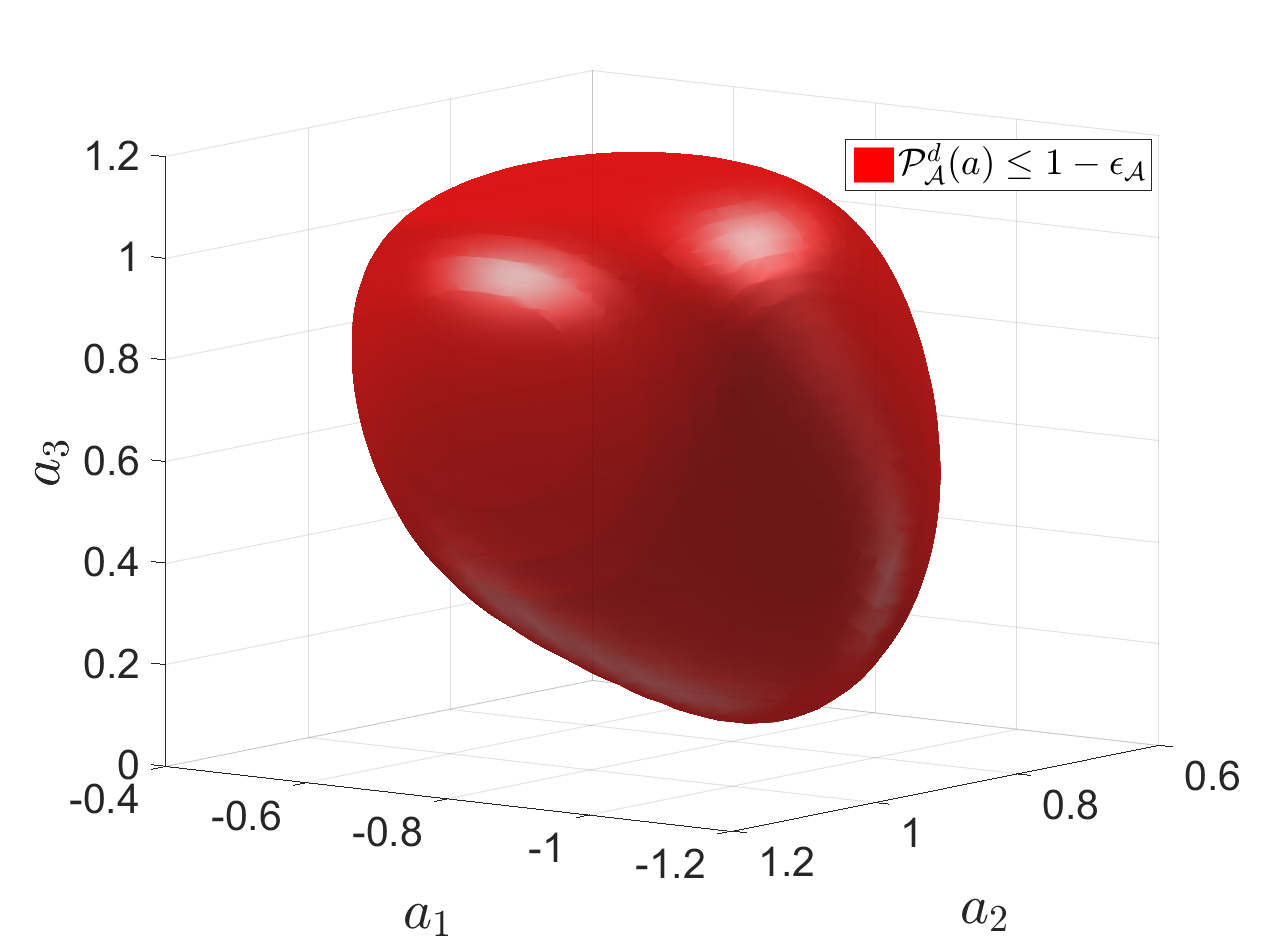}
	\caption{The set $\left\{ (a_1,a_2,a_3):\  \mathcal{P}^{d}_{\mathcal{A}}(a) \leq 1-\epsilon_{\mathcal{A}} \right\}$ obtained by SDP \eqref{P21d} for $d=10$ and $\epsilon_{\mathcal{A}}=0.02$ }
	\label{fig:6} 
\end{figure}

We take $\mathcal{A}_{10}$ and solve SDP in \eqref{P2M}. Based on moments of Lebesgue measure $\mu_x$ on $\chi=[-1, 1]^2$ we construct the matrices in constraints of SDP \eqref{P2M} in terms of unknown moment vectors $\mathbf{y}\in\reals^{S_{5,2r}}$ and $ \mathbf{y_a}\in\reals^{S_{3,2r}}$. The SDP in \eqref{P2M} with $r=7$ is solved using GloptiPoly. Based on obtained solution for moment vectors, we approximate the $(a_1,a_2,a_3)$ with the first order moments of vector $y_a$ as $(y_{a_{100}},y_{a_{010}},y_{a_{001}}) = (-0.999362,0.853458,0.132566)$. Figure \ref{fig:7} shows the sets $\mathcal{S}_1(a) = \{ x\in \chi :\ 0 \leq V(x,a) \leq 1 \} $ and $\mathcal{S}_2(a) = \{x\in \chi : \dot{V}(x,a) \leq - \epsilon_r \Vert x \Vert_2^2 \} $ for obtained coefficients $a$. For obtained $a$, the set $\mathcal{S}_1(a)$ is subset of the set $\mathcal{S}_2(a)$; hence, is an inner approximation of the ROA set.

\begin{figure}[!h]
	\centering
	\includegraphics[scale=0.26]{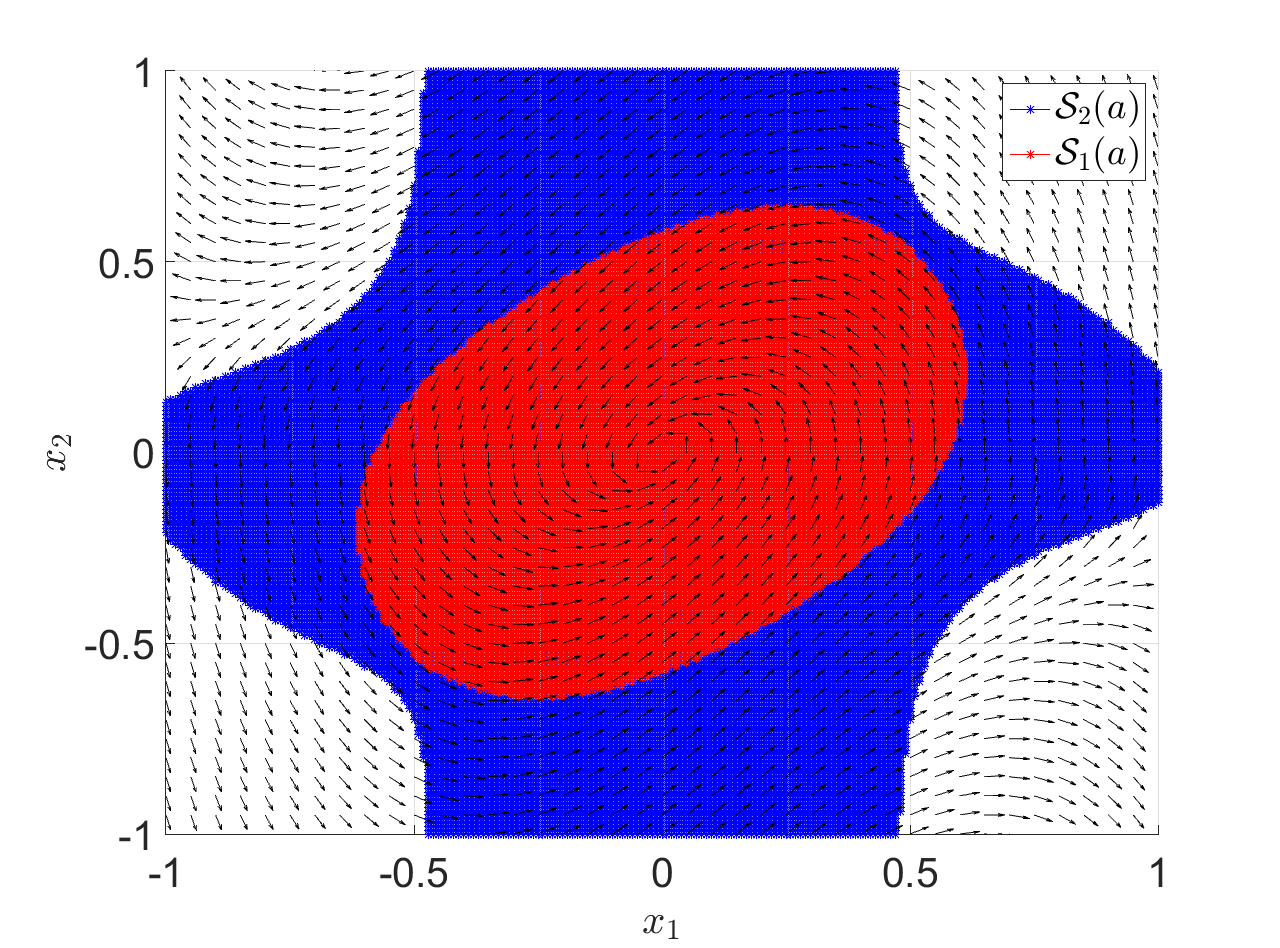}
	\caption{The sets $\mathcal{S}_1(a) = \{ x\in \chi :\ 0 \leq V(x,a) \leq 1 \} $ and $\mathcal{S}_2(a) = \{x\in \chi : \dot{V}(x,a) \leq - \epsilon_r \Vert x \Vert_2^2 \} $ for obtained $a$}
	\label{fig:7} 
\end{figure}

\begin{figure}[!h]
	\centering
	\includegraphics[scale=0.26]{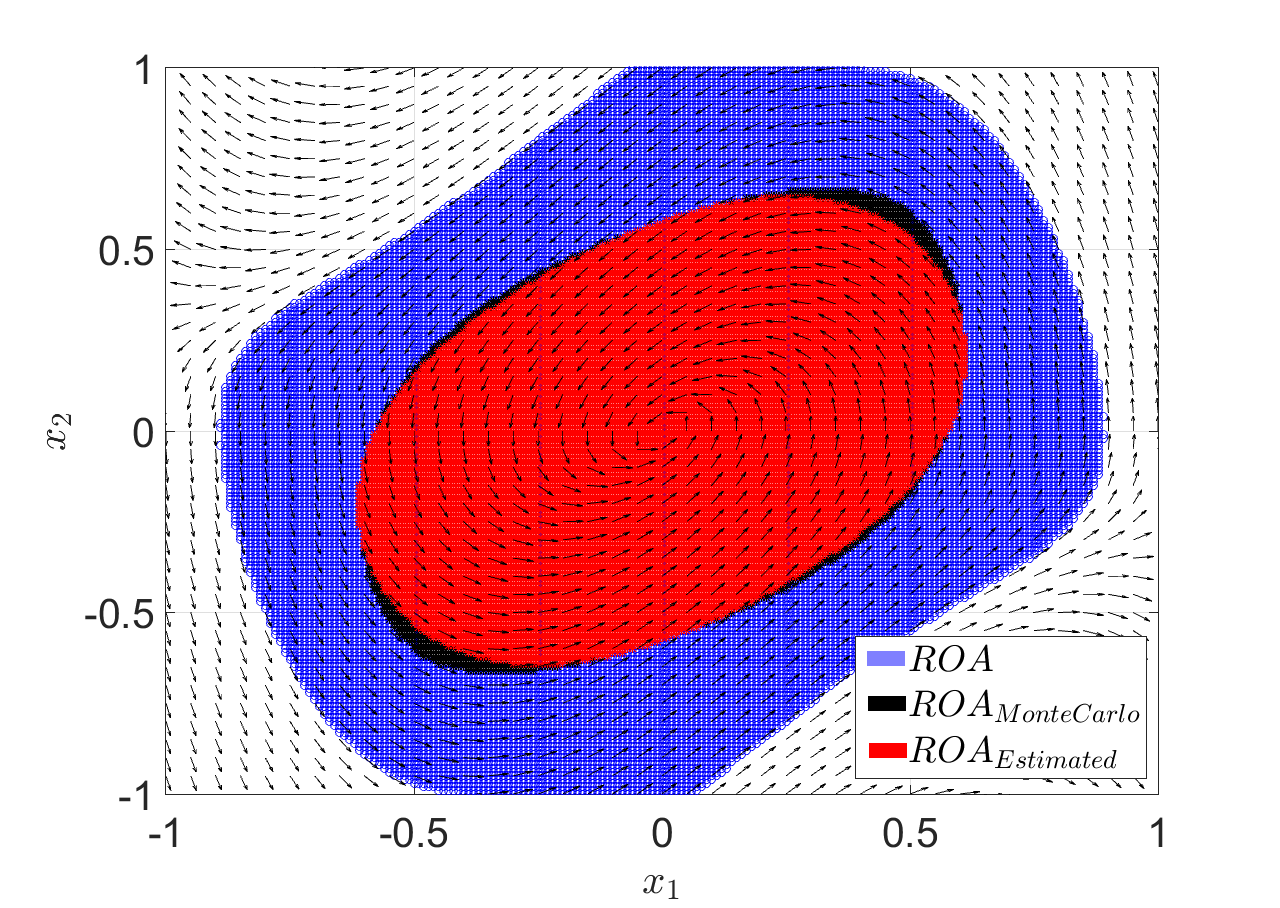}
	\caption{The true and estimated ROA sets}
	\label{fig:8} 
\end{figure}

To test the accuracy of the obtained results, we used Monte Carlo simulation. The obtained result for coefficients of provided Lyapunov function by Monte Carlo method are $(a^*_1,a^*_2,a^*_3)=(-0.6,0,-0.9)$. Figure \ref{fig:8} depicts the true ROA set for the system inside the unite box as well as obtained ROA using Monte Carlo method and convex approach provided in this work.

\subsection{Example 2: Probabilistic Control of Uncertain System}

Consider the uncertain nonlinear system as
\begin{equation}
	\label{eq:control_system}
	\begin{array}{r l}
		x_1(k+1)=&\delta x_2(k),\\
		x_2(k+1)=&x_1(k)~x_3(k),\\
		x_3(k+1)=&x_1(k)-x_2(k)+x_3(k)+u(k)
	\end{array}
\end{equation}
where, initial system states $x_1(0) \sim U[-1,1]$, $x_2(0) \sim U[-1,1]$, $x_3(0) \sim U[-1,1]$, and model parameter $\delta \sim U[-0.2, 0.2]$ are uncertain and uniformly distributed. Also, there is a sphere shaped obstacle centered at $(-0.5,-0.5,0)$ with radius of 0.3 in the state space.

The objective is to find the state feedback control of the form $u(k) = a_1x_1(k) + a_2x_2(k) + a_3x_3(k)$ to lead the states of the system to the cube centered at the origin with the edge length of 0.2 in at most 3 steps and at the same time to avoid the obstacle at each time $k$ with high probability. In other word, we want to maximize the probability of semialgebraic sets $\chi_3 =  \lbrace -0.1 \leq x_1(3) \leq 0.1, \ -0.1 \leq x_2(3) \leq 0.1, \ -0.1 \leq x_3(3) \leq 0.1 \rbrace$ and 
$\chi_{x_k} =  \lbrace (x_1(k)+0.5)^2 + (x_2(k)+0.5)^2 + x_3(k)^2 - 0.3^2 \geq 0 \rbrace, k=1,2 $. 

Hence, the set $\cS_1(a)$ reads as

\begin{footnotesize}
	\begin{equation}
		\mathcal{S}_1(a) = \left\lbrace  (x_0, \delta): \{-0.1 \leq \mathcal{P}_{x_i(3)} \leq 0.1 \}_{i=1}^3,\ \left\lbrace (\mathcal{P}_{x_1(k)}+0.5)^2 + (\mathcal{P}_{x_2(k)}+0.5)^2 + \mathcal{P}_{x_3(k)}^2 - 0.3^2 \geq 0 \right\rbrace _{k=1}^{2}  \right\rbrace 
	\end{equation}
\end{footnotesize}where, $\{x_1(k) = \mathcal{P}_{x_1(k)}(x_0, \delta, a)\}_{k=1}^3$, $\{x_2(k) = \mathcal{P}_{x_2(k)}(x_0, \delta, a)\}_{k=1}^3$, and $\{x_3(k) = \mathcal{P}_{x_3(k)}(x_0, \delta, a)\}_{k=1}^3$ are states of the system in terms of control coefficients vector $a$, initial states $x_0$, uncertain parameters $\delta$ that is derived by dynamic of the system given in \eqref{eq:control_system}. The maximum degree of polynomials defining set $\cS_1(a)$ is 8.

Using Monte Carlo method (see \cite{Jasour2015} Section 5.3.1), we obtain the optimal solution as $(a^*_1,a^*_2,a^*_3)=(-0.5,1,-1)$ and the corresponding optimal probability as 1. To obtain an approximate solution, we solve SDP in  in \eqref{P2M}. Note that since $\cS_2  = \chi$, we dont need to solve SDP in \eqref{P21d} to find polynomial $\cP^d_{\cA}(a)$ and hence the only constraints on coefficients $a$ are $\{1-a^2_i \geq 0\}_{i=1}^3$.

Based on moments of uniform measures, we construct the matrices in constraints of SDP \eqref{P2M} in terms of unknown moment vectors $\mathbf{y}\in\reals^{S_{7,2r}}$ and $ \mathbf{y_a}\in\reals^{S_{3,2r}}$. The SDP in \eqref{P2M} with $r=7$ is solved using GloptiPoly. Based on obtained solution for moment vectors, we approximate the $(a_1,a_2,a_3)$ with the first order moments of vector $y_a$ as $(y_{a_{100}},y_{a_{010}},y_{a_{001}}) = (-0.2820, 0.4766, -0.8602)$ and also we approximate the probability with zero moment of vector $y$ as $y_{0000000} = 1$.

Using Monte Carlo method, the true probability for obtained solution $(a_1,a_2,a_3)=(-0.2820, 0.4766, -0.8602)$ is computed as 0.95. To improve the estimated probability and also to estimate the probability of reaching to target set, i.e., $	\hbox{Prob} \left\lbrace  (x_0, \delta): \{-0.1 \leq \mathcal{P}_{x_i(3)} \leq 0.1 \}_{i=1}^3 \right\rbrace  $ and probability of avoiding the obstacle, i.e., $	\hbox{Prob} \left\lbrace  (x_0, \delta): \left\lbrace (\mathcal{P}_{x_1(k)}+0.5)^2 + (\mathcal{P}_{x_2(k)}+0.5)^2 + \mathcal{P}_{x_3(k)}^2 - 0.3^2 \geq 0 \right\rbrace _{k=1}^{2}  \right\rbrace $ separately, we solve the SDP (3.10) suggested in \cite{Jasour2015} for obtained points $(a_1,a_2,a_3)=(-0.2820, 0.4766, -0.8602)$. 
By solving the SDP with relaxation order of 7 for the set  $	\left\lbrace  (x_0, \delta): \{-0.1 \leq \mathcal{P}_{x_i(3)} \leq 0.1 \}_{i=1}^3 \right\rbrace  $, the estimated probability of 0.994 is obtained, while the true one computed by Monte Carlo is 0.95. Also, the estimated probabilty for the set $ \left\lbrace  (x_0, \delta): \left\lbrace (\mathcal{P}_{x_1(k)}+0.5)^2 + (\mathcal{P}_{x_2(k)}+0.5)^2 + \mathcal{P}_{x_3(k)}^2 - 0.3^2 \geq 0 \right\rbrace _{k=1}^{2}  \right\rbrace $ is obtained as 1, while the true one computed by Monte Carlo is 1. Figure \ref{fig:9} shows the trajectories of the uncertain system \eqref{eq:control_system} controlled by obtained state feedback $u(k) = -0.2820x_1(k) + 0.4766x_2(k) -0.8602x_3(k)$ for different initial points. Note that, although states of the system $x(k)$ avoids the obstacle, the trajectories between the points $x(k), k=0,...,3$ may collide the obstacle.\\

\begin{figure}[!h]
	\centering
	\includegraphics[scale=0.26]{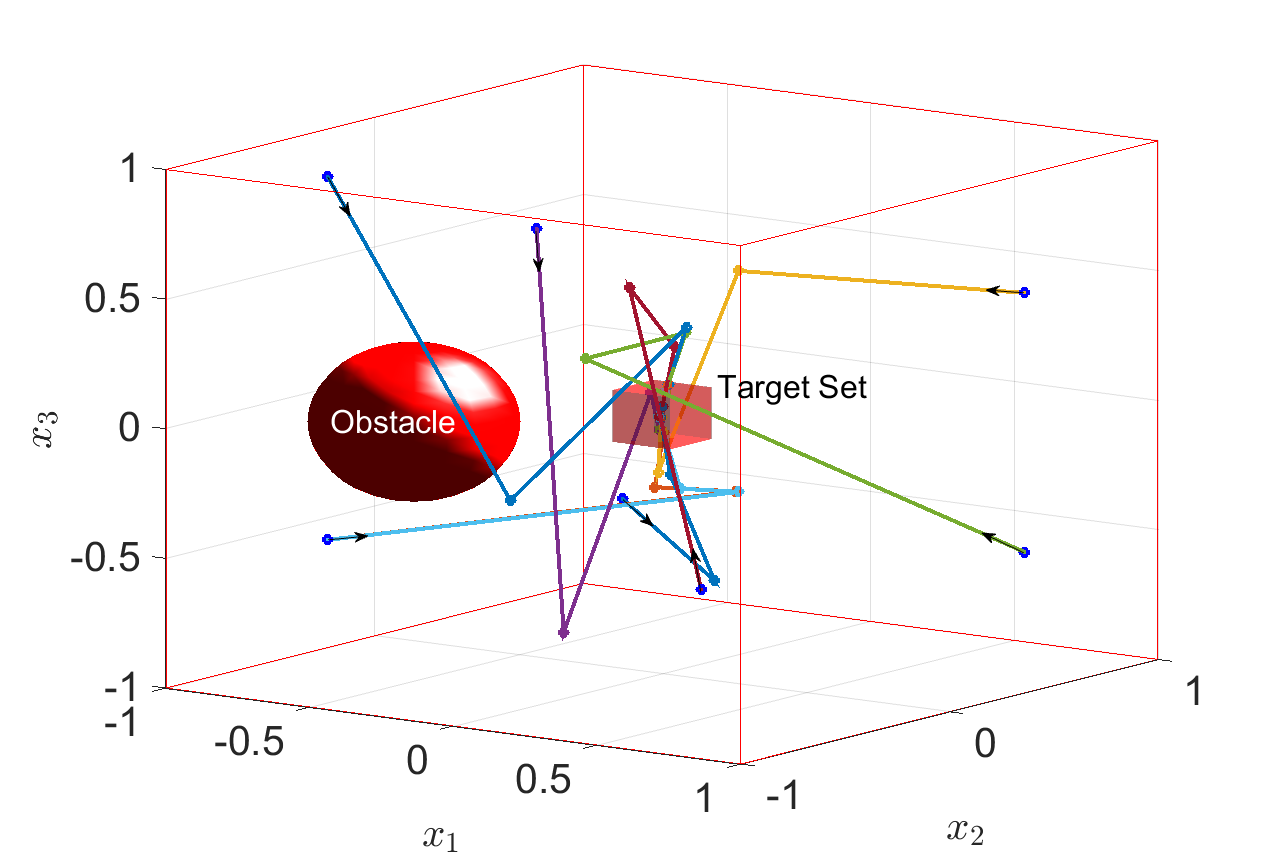}
	\caption{The trajectories of the uncertain system controlled by obtained state feedback}
	\label{fig:9} 
\end{figure}

\begin{remark}
	\label{Large}
	To be able to solve much larger size SDPs in \eqref{P2M}, a first-order augmented Lagrangian algorithm can be implemented, (\cite{Aybat2013, Aybat2014}). As we showed in \cite{Jasour2015}, this algorithm can handle much larger size SDPs than interior point methods can deal with. Also, recently SDSOS optimization technique is introduced to solve large-scale SOS problems as in SDP \eqref{PD2F} and \eqref{P21d}, which results in second order cone programs, see (\cite{Ahmadi2014, Majumdar2014a, Ahmadi2016}) for more information and regarding software. Also, in this paper all polynomials are expanded in the usual monomial basis, and the SDPs are therefore formulated as optimization problems over ordinary monomial moments. However, one can improve the numerical performance by using an orthogonal basis. See (\cite{Jasour2015}, section 3.5), where we employ orthogonal Chebyshev polynomials and redefine moment and localization matrices under the orthogonal basis.
\end{remark}


\section{Conclusion}\label{sec:Con}

In this paper, “constrained volume optimization” problems are introduced, where one aims at maximizing the volume of a set defined by polynomial inequalities such that it is contained in other semialgebraic set. We showed that many nonconvex problems in system and control can be reformulated as constrained volume optimization problems. To be able to obtain a equivalent convex problem, the results from theory of measure and moments as well as duality theory are used. Sequence of semidefinite relaxations is provided whose sequence of optimal values is shown to converge to the optimal value of the original problem. Numerical examples are provided that show that one can obtain reasonable approximations to the optimal solution.



\section{Appendix A: Proof of Theorem \ref{Theo 20}}\label{Appen_Theo20}
Let $\mathcal{A}_{\mathcal{F}}$ be the set of all parameters $a \in \mathcal{A}$ for which the set $\mathcal{S}_1(a)$ is a subset of the set $ \mathcal{S}_2(a)$.
Then, consider the following problem over the measures $\mu_a$ 
\begin{small}
	\begin{align}\label{eq:aux0_problem}
		\mathbf{P_{\mu_a}}:=& \sup_{\mu_a\in\cM_+(\mathcal{A}_{\mathcal{F}})}\left\{ \int_{\mathcal{A}} \mbox{vol}_{\mu_x} (\mathcal S_1(a))~d \mu_a:  \mu_a(\mathcal{A_{\mathcal{F}}})=1 \right\} 
	\end{align}
\end{small}

We first want to show that $\mathbf{P_{vol}^*} = \mathbf{P_{\mu_a}}$. Let $\mu_a$ be a feasible solution to \eqref{eq:aux0_problem}. Since, $\mbox{vol}_{\mu_x}(\mathcal S_1(a)) \le \mathbf{P_{vol}^*}$ for all $a \in \mathcal{A}$, we have $\int_{\mathcal{A}} \mbox{vol}_{\mu_x}(\mathcal S_1(a)) ~d\mu_a \le \mathbf{P_{vol}^*}$. Thus, $\mathbf{P_{\mu_a}} \le \mathbf{P_{vol}^*}$. Conversely, let $a\in\mathcal{A}$ be a feasible solution to the problem in \eqref{intro_P1}; hence, $a$ belongs to the set $\mathcal{A}_{\mathcal{F}}$. Let $\delta_{a}$ denotes the Dirac measure at $a$. The objective value of $a$ in \eqref{intro_P1} is equal to $\mbox{vol}_{\mu_x} (\mathcal S_1(a))$. 
Moreover, $\mu_a = \delta_{a}$ is a feasible solution to the problem in \eqref{eq:aux0_problem} with objective value equal to $\mbox{vol}_{\mu_x} (\mathcal S_1(a))$. This implies that $\mathbf{P_{vol}^*} \le \mathbf{P_{\mu_a}}$. Hence, $\mathbf{P_{vol}^*} = \mathbf{P_{\mu_a}}$, and \eqref{eq:aux0_problem} can be rewritten as
\begin{align} \label{max_p20_proof_p2_2}
	\mathbf{P^*_{vol}}=\sup_{\mu_a\in\cM_+(\mathcal{A}_{\mathcal{F}})}\left\{\int_{\mathcal{A}} \int_{\mathcal{S}_1(a)} d\mu_x d\mu_a:\ \mu_a(\mathcal{A}_{\mathcal{F}})=1\right\} \\
	= \sup_{\mu_a\in\cM_+(\mathcal{A}_{\mathcal{F}})}\left\{\int_{\mathcal K_1} d\mu_a \mu_x:\ \mu_a(\mathcal{A}_{\mathcal{F}})=1\right\}
\end{align}
and using the Lemma \ref{sec2:lem6}, we obtain 
\begin{align} \label{max_p20_proof_p2_3}
	\mathbf{P^*_{vol}}=\sup_{\mu_a\in\cM_+(\mathcal{A}_{\mathcal{F}}),\mu\in\cM_+(\cK_1)} \int d\mu  \\
	\quad \hbox{s.t.}\quad \mu \preccurlyeq \mu_a \times \mu_x,\ \mu_a(\mathcal{A}_{\mathcal{F}})=1. \subeqn
\end{align}

Now, consider two sets $A_1$ and $A_2$. The set $A_1$ is subset of the set $A_2$, if the set $A_1 \cap \bar{A_2}$ is an empty set. In the same way, for a given $a \in \mathcal{A}$, the set $\mathcal{S}_1(a)$ is subset of the set $\mathcal{S}_2(a)$, if the set $\{ x \in \chi:\  (x,a)$ $ \in \mathcal{K}_1 \cap \overline{\mathcal{K}_2} \}$ is an empty set. Hence, for any measure $\mu_a\in\cM_+(\mathcal{A}_{\mathcal{F}})$, we have $\mu_a \times \mu_x \in \cM_+(\overline{\mathcal K_1} \cup \mathcal K_2)$ considering that $\mu_x$ is supported on $\chi$. Therefore, in problem \eqref{max_p20_proof_p2_3} we can look for $\mu_a$ that is supported on $\cA$ such that measure $\mu_a \times \mu_x$ is supported on $\overline{\mathcal K_1} \cup \mathcal K_2$. This results in the following problem:
\begin{align} \label{max_p20_proof_p2_4}
	\mathbf{P^*_{vol}} =&\sup_{\mu_a\in\cM_+(\mathcal{A}),\mu\in\cM_+(\cK_1)} \int d\mu  \\
	\quad \hbox{s.t.}\quad & \mu \preccurlyeq \mu_a \times \mu_x,\ \mu_a(\mathcal{A})=1, \subeqn \\
	& \mu_a \times \mu_x \in \cM_+(\overline{\mathcal K_1} \cup \mathcal K_2). \subeqn
\end{align}
Therefore, $\mathbf{P_{vol}^*} = \mathbf{P_{measure}^*}$.

\section{Appendix B: Proof of Theorem \ref{Theo 1}}\label{Appen_Theo1}

We want to obtain the set $\mathcal{A}_{\mathcal{F}}$ in \eqref{AF}, set of all possible decision parameters $a \in \mathcal{A}$ for which $\mathcal{S}_1(a) \subseteq \mathcal{S}_2(a)$. The idea is to approximate the indicator function of the set $\mathcal{A}_{\mathcal{F}}$; i.e., $\mathcal{I}_{\mathcal{A}_\mathcal{F}}(a) = 1$ if $a \in \mathcal{A}_\mathcal{F}$ and 0 otherwise, with continuous functions (\cite{Henrion2009}, Section 3.2). There exist a sequence of functions $f_i \in \mathcal{C}[a]$ that converges from above to the indicator function of set $\mathcal{A}_{\mathcal{F}}$ as $i \rightarrow \infty$ (\cite{ASH1972}, Theorem A6.6, Urysohn’s Lemma A4.2), which results in an outer approximation of the set $\mathcal{A}_{\mathcal{F}}$ as $\mathcal{A}_{f_i} = \left\{ a\in\mathcal{A}:\  f_i(a) \geq 1 \right\} \supset \mathcal{A}_{\mathcal{F}}$.

To avoid the outer approximation and obtain the inner approximation of the set $\mathcal{A}_{\mathcal{F}}$ instead, we obtain the outer approximation of the complement set $\overline{\mathcal{A}_{\mathcal{F}}}$ by approximating its indicator function. Hence, if $f \in \mathcal{C}$ approximates the indicator function of $\overline{\mathcal{A}_{\mathcal{F}}}$ from above, the set $\mathcal{A}_{f} = \left\{ a\in\mathcal{A}:\  f(a) < 1 \right\} $ is an inner approximation of the set $\mathcal{A}_{\mathcal{F}}$.

To find such function $f(a)$, we use the LP in \eqref{P21}. For a given $a \in \mathcal{A}$, the set $\mathcal{S}_1(a)$ is subset of the set $\mathcal{S}_2(a)$, if the set $\left\{ x \in \chi:\  (x,a) \in \mathcal{K}_1 \cap \overline{\mathcal{K}_2} \right\}$ is an empty set. Hence, the set $\mathcal{A}_{\mathcal{F}}$ can be describe as $\mathcal{A}_{\mathcal{F}}= \{a\in \mathcal{A}:\ \nexists  x\in\chi \hbox{ s.t. } (x,a)\in\cK_1 \cap \overline{\cK_2} \} $. As a result, the complement set $\overline{\mathcal{A}_{\mathcal{F}}}$ read as 
\begin{equation}\label{AFN}
	\overline{\mathcal{A}_{\mathcal{F}}}= \{a\in \mathcal{A}:\ \exists  x\in\chi \hbox{ s.t. } (x,a)\in\cK_1 \cap \overline{\cK_2} \}
\end{equation}
Therefore, to approximate the indicator function of $\overline{\mathcal{A}_{\mathcal{F}}}$ in \eqref{AFN}, the continuous function $f(a)$ should be greater 1 over the set $\left\{ (x,a) \in \mathcal{K}_1 \cap \overline{\mathcal{K}_2} \right\}$ and 0 otherwise, as in \eqref{P21_1} and \eqref{P21_2}, the constrains of the LP. By minimizing the $L_1$-norm of $f(a)$ as in the objective function of \eqref{P21}, we converge to the indicator function of the set $\overline{\mathcal{A}_{\mathcal{F}}}$ from the above. Hence, $\mathcal{A}_{f_i}$ can arbitrarily approximate the set $\mathcal{A}_{\mathcal{F}}$ in \eqref{AF} and (i) and (ii) hold true.


\section{Appendix C: Proof of Theorem \ref{Theo 2}}\label{Appen_Theo2}

First, consider the following problem over the measures supported in the set  $\mathcal{A}_{\mathcal{F}}$ in \eqref{AF}.
\begin{small}
	\begin{equation}\label{eq:aux_problem}
		\mathbf{P_{\mu_a}}:=\sup_{\mu_a\in\cM_+(\mathcal{A}_{\mathcal{F}})}\left\{ \int_{\mathcal{A}} \mbox{vol}_{\mu_x} (\mathcal S_1(a))~d \mu_a:\ \mu_a(\mathcal{A}_{\mathcal{F}})=1\right\}
	\end{equation}
\end{small}
As in Appendix in \ref{Appen_Theo20}, we can show that $\mathbf{P_{vol}^*} = \mathbf{P_{\mu_a}}$ and Eq. \eqref{max_p2_proof_p2_3} is true.
\begin{align} \label{max_p2_proof_p2_3}
	\mathbf{P^*_{vol}}=\sup_{\mu_a\in\cM_+(\mathcal{A}_{\mathcal{F}}),\mu\in\cM_+(\cK_1)} \int d\mu  \\
	\quad \hbox{s.t.}\quad \mu \preccurlyeq \mu_a \times \mu_x,\ \mu_a(\mathcal{A}_{\mathcal{F}})=1. \subeqn
\end{align}

Now, we replace the set $\mathcal{A}_{\mathcal{F}}$ in problem \eqref{max_p2_proof_p2_3} with the set $\mathcal{A}_{f_i}$, where results in Problem \eqref{P2}. Hence, as the set $\mathcal{A}_{f_i} = \left\{ a\in\mathcal{A}:\  f_i(a) < 1 \right\}$ defined in Theorem \ref{Theo 1} converges to the set $\mathcal{A}_{\mathcal{F}}$, the optimal value $\mathbf{P_{f_i}^*}$ converges to the $\mathbf{P_{vol}^*}$ and measures $\mu^*_a(f_i)$ and $\mu^*(f_i)$ converge to $\mu_a = \delta_{a^*}$, Dirac measure at $a^*$, and $\mu = \delta_{a^*} \times \mu_x$, respectively. Also, since $\mathcal{A}_{f_i}$ is an inner approximation of the set $\mathcal{A}_{\mathcal{F}}$, any point $a_i$ in the support of measure $\mu^*_a(f_i)$ is also contained in the set $\mathcal{A}_{\mathcal{F}}$; Hence is an optimal solution to \eqref{intro_P1} and converges to the $a^*$.



\section{Appendix D: Proof of Theorem \ref{Theo 3}}\label{Appen_Theo3}

To drive a finite convex relaxation of the infinite LP problem in \eqref{P21}, we use finite order polynomial $\mathcal{P}^d_{\mathcal{A}}(a)$ to approximate the continuous function $f(a)$ and SOS relaxations to satisfy the constraints of the problem in \eqref{P21}, (\cite{Henrion2009} Section 3.3, \cite{Dabbene2015}). Based on Stone-Weierstrass Theorem \cite{ASH1972}, every continuous function can be uniformly approximated as closely as desired by a polynomial. Then, to make such polynomial $\mathcal{P}^d_{\mathcal{A}}$ to satisfy the constraints of the problem in \eqref{P21}, SOS relaxations are used as in Lemma \ref{sec2:lem7}. Constraints \eqref{P21d_1} implies that $\mathcal{P}^d_{\mathcal{A}}(a)-1$ belongs to the quadratic module generated by polynomials of set $\mathcal{K}_1 \cap \overline{\mathcal{K}_2}$\begin{footnotesize} $ =\left\lbrace (x,a): \cup_{i=1}^{o_2}\{ -\mathcal{P}_{2i} > 0 , \mathcal{P}_{1j}\geq 0, j=1,\dots ,o_1 \}  \right\rbrace$ \end{footnotesize}; hence, $\mathcal{P}^d_{\mathcal{A}}(a)-1$ is nonnegative on the set $\mathcal{K}_1 \cap \overline{\mathcal{K}_2}$. Also, constraint \eqref{P21d_2} implies that $\mathcal{P}^d_{\mathcal{A}}(a)$ belong to the quadratic module generated by polynomials describing the hyper cube $[-1, 1]^n\times[-1, 1]^m$, so $\mathcal{P}^d_{\mathcal{A}}(a)$ is nonnegative over the set $\chi \times \mathcal{A} = [-1, 1]^n\times[-1, 1]^m$. Therefore, by minimizing the $L_1$-norm of the $\mathcal{P}^d_{\mathcal{A}}$ similar to \eqref{P21}, and $d \rightarrow \infty$, $\mathcal{P}^d_{\mathcal{A}}$ converges to the indicator function of the set $\overline{\mathcal{A}_{\mathcal{F}}}$. Therefore, the set $\mathcal{A}_{{d}} = \left\{ a\in\mathcal{A}:\  \mathcal{P}^d_{\mathcal{A}}(a) < 1 \right\}$ converges to the set $\mathcal{A}_{\mathcal{F}}$ in \eqref{AF} as in the Theorem \ref{Theo 1}.

\section{Appendix E: Proof of Theorem \ref{Theo 5}}\label{Appen_Theo5}

The LP in \eqref{P2} can be rewritten as 
\begin{align}
	\mathbf{P_{1}^*}:=&\ \sup \langle \gamma , c \rangle \label{PD2}\\
	&\hbox{s.t.}\quad A^*\gamma = b \label{PD2_1}\subeqn\\
	&\gamma \in \mathcal{M}_+(\mathcal{K}_1)\times \mathcal{M}_+(\mathcal{A}_{f_i}) \label{PD2_2}.\subeqn
\end{align}
where, $\gamma := (\mu,\mu_a) \in \mathcal{M}_+(\mathcal{K}_1)\times \mathcal{M}_+(\mathcal{A}_{f_i})$ is the variable vector, and $c := (1,0) \in \mathcal{C}_+(\mathcal{K}_1) \times \mathcal{C}_+(\mathcal{A}_{f_i})$, so objective function is $\langle \gamma , c \rangle = \int d\mu$. Also, $A^*: \mathcal{M}_+(\mathcal{K}_1)\times \mathcal{M}_+(\mathcal{A}_{f_i}) \rightarrow \mathcal{M}_+(\chi \times \cA)\times \reals_+$ is the linear operator that is defined by $A^*\gamma := (\mu - \mu_a \times \mu_x, \int_{\mathcal{A}}d\mu_a )$ and $b := (0,1) \in \mathcal{M}_+(\chi \times \cA)\times \reals_+ $,(\cite{Henrion2014}, Theorem 2, \cite{anderson1987linear,A.Barvinok2002}).
The problem in \eqref{PD2} is infinite LP defined in cone of nonnegative measures. The cone of nonnegative continuous functions are dual to cone of nonnegative measures.
Based on standard results on LP the dual problem of \eqref{PD2} reads as
\begin{align}
	\mathbf{P_{2}^*}:=&\ \inf \langle b , z \rangle \label{PD3}\\
	&\hbox{s.t.}\quad Az - c \in \mathcal{C}_+(\mathcal{K}_1)\times \mathcal{C}_+(\mathcal{A}_{f_i})\label{PD3_1}\subeqn
\end{align}
where, $z := (\mathcal{W}(x,a),\beta) \in \mathcal{C}_+(\chi \times \cA) \times \reals_+$ is the variable vector, so the objective function is $\langle b , z \rangle = \beta$. The linear operator $A: \mathcal{C}_+(\chi \times \cA) \times \reals_+ \rightarrow \mathcal{C}_+(\mathcal{K}_1) \times \mathcal{C}_+(\mathcal{A}_{f_i})$ satisfies adjoint relation$\langle A^*\gamma,z \rangle = \langle \gamma, Az \rangle $; hence, is defined by $Az := (\mathcal{W}(x,a), \beta - \int_{\mathcal{A}} \mathcal{W}(x,a) d\mu_x )$. As a result, the dual problem \eqref{PD3} is equal to the problem \eqref{PD}.

If problem in \eqref{PD2} is consistent with finite value and the set $$D := \left\lbrace  (A^*\gamma,\langle \gamma,c\rangle) : \gamma \in \mathcal{M}_+(\mathcal{K}_1)\times \mathcal{M}_+(\mathcal{A}_{f_i}) \right\rbrace $$ is closed, then there is no duality gap between \eqref{PD2} and \eqref{PD3}. The support of measures in \eqref{PD2} are compact. Also, the measure $\mu$ is constrained by the measure $\mu_a \times \mu_x$ in which, measure $\mu_a$ is probability measure; i.e., $\mu_a(\mathcal{A}_{f_i})=1$, and $\mu_x$ is finite Borel measure defined on compact set $\chi$. Hence, $\mathbf{P^*_1} = \sup \int d\mu < \infty$. Also, the feasible set of \eqref{PD2} is nonempty for instance $(\delta_a\times\mu_x,\delta_a)$ for $a \in \mathcal{A}_{f_i}$ is a feasible solution; Therefor $0 \leq \mathbf{P^*_1} = \sup \int d\mu < \infty$.
Using sequential Banach$ - $Alaoglu theorem \cite{ASH1972} and weak-$\star$ continuity of the $A^*$, there exist an accumulation point of $\gamma_k = (\mu_k,{\mu_a}_k)$ in the weak-$\star$ topology of nonnegative measures such that $\hbox{lim}_{k \rightarrow \infty} \left(  (A^*\gamma_k,\langle \gamma_k,c\rangle) \right) \in D$; hence, $D$ is closed, (\cite{Henrion2014}, Theorem 2).
\section{Appendix F: Proof of Theorem \ref{Theo 6}}\label{Appen_Theo6}
For simplicity, we denote the polynomials $\left(\{1-\mathcal{P}^d_{\mathcal{A}}\},\{(1-a^2_i)\}_{i=1}^m \right)$ that construct the set $\cA_d$ by $\{\mathcal{P}_{\mathcal{A}j}(a)\}_{j=1}^{o_a}$. Matrices of the problem \eqref{P2M} can be rewritten as follow, (\cite{Henrion2009}). {\small $M_r(\mathbf y)= \sum_{\alpha}A_{\alpha}y_{\alpha}$} and {\small $M_{r-r_j}(\mathbf{y} ; \mathcal{P}_{1j})=\sum_{\alpha}B^j_{\alpha}y_{\alpha}$}.
Also, {\small $M_r ({\mathbf y}_{\mathbf a})=\sum_{\alpha}D_{\alpha}{y_a}_{\alpha}$}, {\small $M_{r-r_a}(\mathbf{y}_{\mathbf a}; \mathcal{P}_{\mathcal{A}j}(a))=\sum_{\alpha}E^j_{\alpha}{y_a}_{\alpha}$}, and {\small $M_r (\mathbf{y_a}\times\mathbf{y_x}-{\mathbf y})=\sum_{\alpha}F_{\alpha}{y_a}_{\alpha}- \sum_{\alpha}A_{\alpha}y_{\alpha}$} for appropriate real symmetric matrices $(A_{\alpha},\{B^j_{\alpha}\}_{j=1}^{o_1},D_{\alpha},\{E^j_{\alpha}\}_{j=1}^{o_a},F_{\alpha})$ and $0 \leq |\alpha| \leq 2r$. Let, $\gamma=(\mathbf{y}\in\reals^{S_{n+m,2r}},\ \mathbf{y_a}\in\reals^{S_{m,2r}})$. Then problem in \eqref{P2M} can be rewritten as a standard form as follow:
\begin{small}
	\begin{align}
		\mathbf{P^*_r}:= &\sup_{\gamma} b^T\gamma,
		\label{PD3F}\\
		\hbox{s.t.}\quad & C_1 + \sum_{\alpha}\hat{A}_{\alpha}\gamma_{\alpha} \succcurlyeq 0, \label{PD3F_1}\subeqn\\
		& C^j_2 + \sum_{\alpha}\hat{B}^j_{\alpha}\gamma_{\alpha} \succcurlyeq 0,\ \  j=1,\dots,o_1 \label{PD3F_2}\subeqn\\
		& C_3 - \sum_{\alpha}\hat{C}_{\alpha}\gamma_{\alpha} \succcurlyeq 0, \label{PD3F_33}\subeqn\\
		& C_4 + \sum_{\alpha}\hat{D}_{\alpha}\gamma_{\alpha} \succcurlyeq 0, \label{PD3F_3}\subeqn\\
		& C^j_5 + \sum_{\alpha}\hat{E}^j_{\alpha}\gamma_{\alpha}\succcurlyeq 0, \ \ j=1,\dots,o_a \label{PD3F_4}\subeqn\\
		& C_6 + \sum_{\alpha}\hat{F}_{\alpha}\gamma_{\alpha}\succcurlyeq 0, \label{PD3F_5}\subeqn
	\end{align}
\end{small}where, $b = (1,\textbf{0}) \in \reals^{S_{n+m,2r}+S_{m,2r}}$, $(C_1,C_2,C_4,C_5,C_6)$ are zero matrices, $(\hat{A}_{\alpha},\{\hat{B}^j_{\alpha}\}_{j=1}^{o_1}$ $,\hat{D}_{\alpha},$ $\{\hat{E}^j_{\alpha}\}_{j=1}^{o_a},\hat{F}_{\alpha})$ are real symmetric matrices, $C_3=1$, and $\hat{C}^T=(\textbf{0} \in \reals^{S_{n+m,2r}},1,\textbf{0} \in \reals^{S_{m,2r}-1}) \in \reals^{S_{n+m,2r}+S_{m,2r}}$.
Based on standard results on duality of SDP, the dual problem to \eqref{PD3F} reads as
\begin{footnotesize}
	\begin{align}
		\mathbf{P^*_d}:= &\inf_{\{X^j\}_{j=0}^{o_1},\{Y^j\}_{j=0}^{o_a},Z,\beta} \left\langle  C_1,X^0 \right\rangle
		+ \sum_{j=1}^{o_1} \left\langle  C^j_2,X^j \right\rangle + \left\langle  C_3,\beta \right\rangle +  \left\langle  C_4,Y^0 \right\rangle + \sum_{j=1}^{o_a} \left\langle  C^j_5,Y^j \right\rangle +  \left\langle  C_6,Z \right\rangle \label{PD4F}\\
		\hbox{s.t.}\quad & \beta - \left\langle A_{\alpha},X^0 \right\rangle - \sum_{j}^{o_1} \left\langle B^j_{\alpha},X^j \right\rangle - \left\langle D_{\alpha},Y^0 \right\rangle - \sum_{j}^{o_a} \left\langle E^j_{\alpha},Y^j \right\rangle- \left\langle F_{\alpha},Z \right\rangle = b_{\alpha},\ \ \alpha = 0,  \label{PD4F_1}\subeqn\\
		& - \left\langle A_{\alpha},X^0 \right\rangle - \sum_{j}^{o_1} \left\langle B^j_{\alpha},X^j \right\rangle - \left\langle D_{\alpha},Y^0 \right\rangle - \sum_{j}^{o_a} \left\langle E^j_{\alpha},Y^j \right\rangle- \left\langle F_{\alpha},Z \right\rangle = b_{\alpha},\ \ 0 < |\alpha| \leq 2r,  \label{PD4F_11}\subeqn\\
		& X^0,\{X^j\}_{j=1}^{o_1},Y^0,\{Y^j\}_{j=1}^{o_a},Z, \beta \succcurlyeq 0 \label{PD4F_2}\subeqn
	\end{align}
\end{footnotesize}where, $\left\langle X,Y \right\rangle = \hbox{trace}(XY)$. This problem is equal to the problem in \eqref{PD2F}. Based on the defined matrices and vectors, the cost function of \eqref{PD4F} is equal to $\beta$. Also, let $\cB_d$ denote the vector comprised of the monomial basis of $\mathbb R_{\rm d}[a,x]$. We can represent the polynomials of \eqref{PD2F} as  
$\mathcal{P}^d_{\mathcal{W}}(x,a) = \cB_d^T X^0 \cB_d$, $\mathcal{QM} \left( \{\mathcal{P}_{1j}\}_{j=1}^{o_1} \right) = \sum_{j}^{o_1} \cB_d^T X^j \cB_d$, $\int \mathcal{P}^d_{\mathcal{W}}(x,a) d\mu_x = \cB_d^T Y^0 \cB_d$, $\mathcal{QM} \left(\{\mathcal{P}_{\mathcal{A}j}\}_{j=1}^{o_a} \right) = \sum_{j}^{o_a} \cB_d^T Y^j \cB_d$, and $\hat{\mathcal{P}}^d_{\mathcal{W}}(x,a)$ $= \cB_d^T Z \cB_d$. Then constraints \eqref{PD4F_1} and \eqref{PD4F_11} are conditions for $\alpha$-th coefficient of polynomial $\mathcal{P}^d_{\mathcal{W}}(x,a)$ so that as constraints \eqref{PD2F_1} and \eqref{PD2F_2}, $\mathcal{P}^d_{\mathcal{W}}(x,a)-1 \in \mathcal{QM} \left( \{\mathcal{P}_{1j}\}_{j=1}^{o_1} \right)$, $\beta - \hat{\mathcal{P}}^d_{\mathcal{W}}(x,a)  \in \mathcal{QM} \left(\{\mathcal{P}_{\mathcal{A}j}\}_{j=1}^{o_a} \right)$, and $\hat{\mathcal{P}}^d_{\mathcal{W}}(x,a) = \int_\chi \mathcal{P}^d_{\mathcal{W}}(x,a) d\mu_x$ are satisfied. 

Based om Slater's sufficient condition,if the feasible set of strictly positive matrices in constraint of primal SDP is nonempty, then there is no duality gap. 
Consider SDP in \eqref{P2M}. Let $\mu_a$ uniform measure on $\mathcal{A}_d$ and $\mu = \mu_a \times \mu_x$. Since set $\cK_1$ and $\mathcal{A}_d$ have a nonempty interior, then $M_r(\mathbf y)\succ 0$, $M_{r-r_j}(\mathbf{y}; \mathcal{P}_{1j}) \succ 0, j=1,\dots ,o_1 $, $ M_r ({\mathbf y}_{\mathbf a}) \succ 0$, and $M_{r-r_a}(\mathbf{y}_{\mathbf a}; \mathcal{P}_{aj}) \succ 0, j=1,\dots ,o_a$. Based on Assumption \ref{Assum}, $\chi \times \mathcal{A} \setminus \cK_1$ has nonempty interior; hence $M_r (\mathbf{y_a}\times\mathbf{y_x}-{\mathbf y}) \succ 0$. Therefore, Slater's condition holds, (see \cite{Henrion2009} for similar setup). 

	\bibliographystyle{ieeetr}
	
\bibliography{AshkanCollection}

\end{document}

%% file: defs.tex
\def\cA{\mathcal{A}}
\def\cB{\mathcal{B}}

\def\cF{\mathcal{F}}

\def\cK{\mathcal{K}}

\def\cM{\mathcal{M}}

\def\cP{\mathcal{P}}
\def\cQ{\mathcal{Q}}

\def\cS{\mathcal{S}}

\def\cW{\mathcal{W}}

\def\smskip{\smallskip}

\def\texitem#1{\par\smskip\noindent\hangindent 25pt
               \hbox to 25pt {\hss #1 ~}\ignorespaces}

\def\norm#1{\|#1\|}

\newcommand{\BEAS}{\begin{eqnarray*}}
\newcommand{\EEAS}{\end{eqnarray*}}
\newcommand{\BEA}{\begin{eqnarray}}
\newcommand{\EEA}{\end{eqnarray}}
\newcommand{\BEQ}{\begin{eqnarray}}
\newcommand{\EEQ}{\end{eqnarray}}
\newcommand{\BIT}{\begin{itemize}}
\newcommand{\EIT}{\end{itemize}}
\newcommand{\BNUM}{\begin{enumerate}}
\newcommand{\ENUM}{\end{enumerate}}

\newcommand{\BA}{\begin{array}}
\newcommand{\EA}{\end{array}}

\newcommand{\reals}{\mathbb{R}}
\newcommand{\integers}{\mathbb{Z}}

\newif\ifpagenumbering
\pagenumberingtrue

\pagenumberingfalse

\newsavebox{\theorembox}
\newsavebox{\lemmabox}
\newsavebox{\remarkbox}
\newsavebox{\assbox}
\savebox{\theorembox}{\noindent\bf Theorem}
\savebox{\lemmabox}{\noindent\bf Lemma}
\savebox{\remarkbox}{\noindent\bf Remark}
\savebox{\assbox}{\noindent\bf Assumption}
\newtheorem{assumption}{\usebox{\assbox}}
\newtheorem{remark}{\usebox{\remarkbox}}[section]

